\documentclass{article}
\usepackage{graphicx} 

\usepackage{xcolor, amsthm, amsmath, amssymb, amsfonts, enumerate}

\usepackage{xcolor}
\definecolor{nicelavender}{RGB}{153, 128, 250}
\definecolor{nicegreen}{RGB}{106, 176, 76}
\definecolor{random}{RGB}{60, 170, 200}
\definecolor{nicepurple}{RGB}{200, 90, 160}
\definecolor{darkpurple}{RGB}{164, 10, 106}
\definecolor{darkblue}{RGB}{0, 0, 100}
\newcommand\cece[1]{{\color{nicegreen}\textsc{Cece:} #1}}

\newcommand{\rnotes}[1]
{{\color{darkpurple}\textsc{Rebecca:} #1}}

\newcommand{\mc}[1]{\mathcal{#1}}
\newcommand{\mb}[1]{\mathbb{#1}}

\usepackage{thmtools} 
\usepackage{thm-restate}

\usepackage{amssymb}
\usepackage{amsmath}
\usepackage{bbm}

\usepackage{algorithm}
\usepackage{algorithmicx}

\usepackage{tikz}
\usetikzlibrary{calc}

\usepackage{todonotes}

\usepackage{hidde-standard-preamble}

\title{Almost perfect graph classes}

\author[1]{Cicely (Cece) Henderson}
\author[1]{Hidde Koerts}
\author[4]{Taite LaGrange}
\author[1]{Sophie Spirkl\thanks{We acknowledge the support of the Natural Sciences and Engineering Research Council of Canada (NSERC), [funding reference number RGPIN-2020-03912]. Cette recherche a \'et\'e financ\'ee par le Conseil de recherches en sciences naturelles et en g\'enie du Canada (CRSNG), [num\'ero de r\'ef\'erence RGPIN-2020-03912]. This project was funded in part by the Government of Ontario. This research was completed while Spirkl was an Alfred P. Sloan Fellow.  This research was undertaken, in part, thanks to
funding from the Canada Research Chairs Program.
}}
\author[2]{Agnès Totschnig}
\author[1]{Massimo Vicenzo}
\author[3]{Rebecca Whitman}
\affil[1]{Department of Combinatorics and Optimization, University of Waterloo}
\affil[2]{Department of Mathematics, Massachusetts Institute of Technology}
\affil[3]{Department of Mathematics,
University of California, Berkeley}
\affil[4]{David R. Cheriton School of Computer Science, University of Waterloo}

\begin{document}

\maketitle

\begin{abstract}
A graph $G$ is perfect if $\omega(H) = \chi(H)$ for each induced subgraph $H$ of $G$. In 2002, Chudnovsky, Robertson, Seymour, and Thomas famously proved the Strong Perfect Graph Theorem. Motivated by this forbidden induced subgraph characterization of the class of perfect graphs as well as the possible extension of efficient algorithms on perfect graphs, we consider the structure of graphs that are almost perfect. We say a graph is $c$-apex perfect if there is a constant $c$ number of vertices such that, upon the deletion of these vertices, what remains is a perfect graph. In this paper, we characterize the class of the sets of graphs $\mathcal{H}$ with $|\mathcal{H}|\leq 2$ for which there exists $c \in \mathbb{N}$ with the property that each $\mathcal{H}$-free graph is $c$-apex perfect. We also extend these results to several notable subclasses of perfect graphs, including chordal, interval, split, bipartite, and complete multipartite graphs.
\end{abstract}

\section{Background and Definitions}\label{sec:background}

The class of perfect graphs is one of the most fundamental classes in graph theory. A graph $G$ is \emph{perfect} if $\omega(H) = \chi(H)$ for each induced subgraph $H$ of $G$, where $\omega(H)$ is the \emph{clique number}, or size of a largest clique, of $H$, and $\chi(H)$ is the \emph{chromatic number} of $H$. Not only are many computational problems efficiently solvable on perfect graphs, but also perfect graphs themselves play a central role in structural graph theory. See for instance~\cite{golumbic2004algorithmic,grotschel1984polynomial,trotignon2013perfect}. In 1972, Lov\'{a}sz~\cite{lovasz1972characterization} famously provided one of the first major structural characterizations of perfect graphs using graph complements. The \emph{complement} of a graph $G$, denoted $\overline{G}$, is the graph with vertex set $V(\overline{G}):=V(G)$ and $uv\in E(\overline{G})$ if and only if $uv\notin E(G)$.

\begin{theorem}[Perfect Graph Theorem - Lov\'{a}sz~\cite{lovasz1972characterization}]\label{thm:WPGT}
    A graph $G$ is perfect if and only if $\overline{G}$ is perfect.
\end{theorem}

In 1961, Berge~\cite{berge1961farbung} posited that the structure of perfect graphs also relates to the presence of induced subgraphs called holes and antiholes. For a graph $G$ and a set $S \subseteq V(G)$ of vertices, $G[S]$ denotes the subgraph of $G$ \emph{induced} by $S$. We say that $G$ \emph{contains} graph $H$ (as an induced subgraph) if there exists $S \subseteq V(G)$ such that $G[S] \cong H$, where $G \cong H$ denotes that $G$ is isomorphic to $H$. In particular, $G$ contains a \emph{hole} $H$ if $H$ isomorphic to a cycle with at least four vertices. An \emph{antihole} of $G$ is an induced subgraph of $G$ whose complement is isomorphic to a hole. We call $H$ an \emph{odd hole} or \emph{odd antihole}, respectively, if the associated cycle is of odd length. In a groundbreaking result from 2002, Chudnovsky, Robertson, Seymour, and Thomas~\cite{chudnovsky2006strong} confirmed Berge's longstanding conjecture by proving the Strong Perfect Graph Theorem (which also implies \cref{thm:WPGT} as a corollary). 

\begin{theorem}[Strong Perfect Graph Theorem - Chudnovsky, Robertson, Seymour, and Thomas~\cite{chudnovsky2006strong}]\label{thm:SPGT}
    A graph $G$ is perfect if and only if it contains neither odd holes nor odd antiholes.
\end{theorem}

Given the centrality of perfect graphs, it is natural to wonder about the structure of graphs that are \textit{almost} perfect. Others have considered variations of this question (for example, Lov\'{a}sz~\cite{lovasz1972characterization} made a first step towards the Strong Perfect Graph Theorem by considering ``minimally imperfect graphs''). Motivated by the possible extension of efficient algorithms on perfect graphs, a reasonable interpretation of an almost perfect graph is one that is a small number of vertices away from being perfect, that is, a graph for which there exists a constant $c$ number of vertices such that, upon the deletion of these vertices, what remains is a perfect graph. We call such a graph \emph{$c$-apex perfect}. More generally, for a graph property $\mathcal{P}$ and a constant $c \in \mathbb{N}$, we say that a graph $G$ is \emph{$c$-apex-$\mathcal{P}$} if there exists a set $S \subseteq V(G)$ of at most $c$ vertices  such that $G-S$ has property $\mathcal{P}$, where $G - S$ is the graph $G[V(G) \setminus S]$ obtained by removing the vertices in $S$. Similarly, we say a graph class $\mathcal{F}$ is \emph{$c$-apex-$\mathcal{P}$} if each graph $G \in \mathcal{F}$ is $c$-apex-$\mathcal{P}$. If there exists a $c$ such that a class $\mathcal{F}$ is $c$-apex-$\mathcal{P}$, we say that $\mathcal{F}$ is \emph{bounded-apex-$\mathcal{P}$}.

As alluded to above, being bounded-apex-$\mathcal{P}$ has direct algorithmic implications: many problems which are polynomial-time solvable on graphs with property $\mathcal{P}$ remain polynomial-time solvable for classes of bounded-apex-$\mathcal{P}$ graphs. Indeed, one may guess the bounded number of \emph{apices,} or vertices whose removal leaves a graph with property $\mathcal{P}$, as well as information about the solution's restriction to the apices, and efficiently solve the problem on the remainder of the graph. By extending said solution to the apices (and possibly locally modifying the solution), one can then efficiently find either optimal solutions or good approximations for many problems. As an example, Gr\"otschel, Lovasz, and Schrijver \cite{grotschel1984polynomial} gave a polynomial-time algorithm for finding a maximum-weight clique in a perfect graph, which yields an analogous algorithm for $c$-apex-perfect graphs as follows. Guess the set of apices, as well as the subset $X$ of them which appears in our solution. Delete all apices, as well as all vertices with a non-neighbour in $X$; then apply the algorithm from \cite{grotschel1984polynomial} to the remaining graph. Return the union of its output and $X$. 

This meaning of ``apices'' originated in the context of graph minors, as, for instance, in Marx and Schlotter's~\cite{marx2012obtaining} study of planar graphs and Sau, Stamoulis, and Thilikos'~\cite{sau2023k} study of minor closed classes. More recently, this notion has also been considered for classes defined by forbidden induced subgraphs. For example, in~\cite{singh2023apex} Singh, Sivaraman, and Zaslavsky consider the class of 1-apex-cographs (cographs are graphs with no induced four-vertex path). Similarly, in~\cite{borowiecki2018p} Borowiecki, Drgas-Burchardt, and Sidorowicz consider the order of forbidden induced subgraphs of $c$-apex-$\mathcal{P}$ for an arbitrary property $\mathcal{P}$ such that the class is \emph{hereditary}, that is, closed under taking subgraphs. The study of bounded-apex-perfect graphs in particular is natural because of the forbidden induced subgraph characterization of the class of perfect graphs provided by~\cref{thm:SPGT}. 

To that end, we require the following definitions. If graph $G$ does not contain graph $H$, we say that $G$ is \emph{$H$-free}. Similarly, for a set $\mathcal{H}$ of graphs, $G$ is \emph{$\mathcal{H}$-free} if $G$ does not contain any graph in $\mathcal{H}$.

In this paper, we consider the following question, first for perfect graphs and then more generally. 

\begin{question}
\label{question:broadquestion}
    Given a property $\mathcal{P}$, for which sets of graphs $\mathcal{H}$ is the class of $\mathcal{H}$-free graphs bounded-apex-$\mathcal{P}$?
\end{question}

Given the breadth of~\cref{question:broadquestion}, we initiate its study by considering a similar question where $|\mathcal{H}|$ is bounded:

\begin{question}
\label{question:mainquestion}
    Given a property $\mathcal{P}$, for which sets of graphs $\mathcal{H}$ with $|\mathcal{H}| \le 2$ is the class of $\mathcal{H}$-free graphs bounded-apex-$\mathcal{P}$?
\end{question}

Our main result answers \cref{question:mainquestion} for perfect graphs, as follows. 
Note that for graphs $G$ and $H$ and a non-negative integer $r$, $G+H$ denotes the graph isomorphic to the disjoint union of $G$ and $H$; $rH$ denotes the graph isomorphic to the disjoint union of $r$ copies of $H$; and $K_r$, $C_r$ and $P_r$ denote a complete graph, a cycle, and a path on $r$ vertices, respectively.

\begin{restatable}{theorem}{main}\label{thm:main}
    Let $H_1, H_2$ be graphs. There exists a constant $c \in \mathbb{N}$ such that the class of $\{H_1, H_2\}$-free graphs is $c$-apex-perfect if and only if -- up to switching the labels of $H_1$ and $H_2$, and taking complements of both -- one of the following holds: 
    \begin{enumerate}
        \item $H_1$ is an induced subgraph of $P_4$;  \label{case:p4}
        \item $H_1$ is an induced subgraph of $\overline{P_3+K_1}$ and $H_2$ is an induced subgraph of $P_3+rK_1$ for some $r \in \mathbb{N}$;   \label{case:triangle}
        \item $H_1$ is a complete graph and $H_2$ is edgeless; \label{case:ramsey}
        \item $H_1$ is an induced subgraph of $\overline{K_2 + 2K_1}$ and $H_2$ is an induced subgraph of $K_2 + 2K_1$; or \label{case:p3+p1}
        \item $H_1 \cong C_4$ and $H_2 \cong 2K_2$, in which case $c = 1$. \label{case:c42p2}
    \end{enumerate}
\end{restatable}

Many important graph families are subclasses of perfect graphs. Thus, in addition to~\cref{thm:main}, we also answer~\cref{question:mainquestion} for several notable subclasses of perfect graphs, including chordal graphs, interval graphs, split graphs, bipartite graphs, and complete multipartite graphs. The families under study were selected on the basis of their importance in structural and algorithmic graph theory (see \cite{chordal, interval, split, bipartite}). The relevant definitions for each of these families are included in their respective sections.

\subsection{Outline of paper}
We answer \cref{question:mainquestion} for perfect graphs in Sections \ref{sec:bounded_properties} to \ref{sec:main_proof}. First, we begin in \cref{sec:bounded_properties} by discussing necessary conditions on a set $\mathcal{H}$ of graphs for the $\mathcal{H}$-free graphs to be bounded-apex-$\mathcal{P}$, where $\mathcal{P}$ represents being perfect or inclusion in a subclass of perfect graphs. We next consider the characterization of $\mathcal{H}$ where $\mathcal{P}$ represents being perfect. The theorem is stated here and Sections \ref{sec:imperfect_families} to \ref{sec:main_proof} are dedicated to its proof. Specifically, in \cref{sec:imperfect_families} we establish a number of infinite families of graphs that are never bounded-apex-perfect. In \cref{sec:successful_pairs} we show that the given families $\mathcal{H}$ of forbidden induced subgraphs do each imply that the class of $\mathcal{H}$-free graphs are bounded-apex-perfect. We complete the proof of \cref{thm:main} in \cref{sec:main_proof} by demonstrating that no other sets $\mathcal{H}$ of graphs have the desired property.

The remainder of the paper is dedicated to characterizing when forbidding one or two graphs produces a family of bounded-apex-chordal, bounded-apex-interval, bounded-apex-split, bounded-apex-bipartite, or bounded-apex-complete-multipartite graphs. Since each of these families is a subset of the bounded-apex-perfect graphs, we are able to use the machinery established for our proof of \cref{thm:main} as a jumping-off point in the subsequent arguments. 

\subsection{Additional Definitions and Notation}\label{sec:notation}
All graphs considered in this paper are simple and finite. We use $[k]$ to denote the set $\{1, \dots, k\}$. For any two subsets $A,B \subseteq V(G)$, we say that $A$ is \emph{complete} to $B$ if all vertices of $A$ are adjacent to all vertices of $B$. Similarly, we say that $A$ and $B$ are \emph{anti-complete} if there does not exist an edge in $G$ with one endpoint in $A$ and the other endpoint in $B$.

A \emph{clique} in a graph $G$ is a vertex set $S \subseteq V(G)$ such that $G[S]$ is a complete graph.  
A \emph{stable set} in a graph $G$ is a vertex set $S \subseteq V(G)$ such that $G[S]$ is edgeless.
The \emph{stability number} of a graph $G$, denoted by $\alpha(G)$, is the size of a largest stable set in $G$.

Here we introduce several other important graph classes used throughout the paper. 
A \emph{linear forest} is a graph $F$ in which each component of $F$ is a path. A \emph{split graph} is a graph with vertex set $K\cup S$ where the subgraph induced by $K$ is a clique, and the subgraph induced by $S$ is a stable set. Lastly, a graph $G$ is \emph{complete multipartite} if $V(G)$ can be partitioned into sets $V_1,\dots,V_k$, for $k\geq1$ such that for all $u\in V_i$, $v\in V_j$, we have that $uv\in E(G)$ if and only if $i\neq j$. 
A \emph{diamond}, denoted $K_4^-$, is $K_4$ with one edge removed.

\section{Necessary conditions for bounded-apex-perfect classes}
\label{sec:bounded_properties}

In this section, we elucidate a number of properties of bounded-apex-perfect graph classes and consider necessary conditions on a set of graphs $\mathcal{H}$ for the class of $\mathcal{H}$-free graphs to be bounded-apex-perfect. We first observe that since the class of perfect graphs is closed under complementation, so too is the class of $c$-apex-perfect graphs. We record this as a lemma. 

\begin{lemma} \label{lem:complement}
    Let $\mathcal{C}$ be a class of graphs and let $c \in \mathbb{N}$. If $\mathcal{C}$ is $c$-apex-perfect, then so is $\{\overline{G} : G \in \mathcal{C}\}$. 
\end{lemma}
\begin{proof}
    Let $\mathcal{C}$ be a $c$-apex-perfect class of graphs and let $G$ be a graph in $\mathcal{C}$. Then there is some set $S \subset V(G)$ with $|S| \leq c$ such that the graph $G - S$ is perfect. Since the complement of a perfect graph is also perfect by \cref{thm:WPGT}, $\overline{G} - S$ is perfect. Hence $\overline{G}$ is $c$-apex-perfect, for all $ G \in \mathcal{C}$, as desired.
\end{proof}

Next, we show that linear forests are the critical structure for forbidden induced subgraphs of bounded-apex-perfect graphs. (Note that Lemma \ref{lem:LinForest} does not hold for infinite $\mathcal{H}$, as shown by Theorem \ref{thm:SPGT}.)

\begin{lemma}\label{lem:LinForest}
    Let $\mathcal{H}$ be a finite set of graphs.  If the class of $\mathcal{H}$-free graphs is bounded-apex-perfect then some graph in $\mathcal{H}$ is a linear forest, and some graph in $\mathcal{H}$ is the complement of a linear forest. 
\end{lemma}
\begin{proof}
    Assume for the sake of contradiction that there exists a constant $c \in \mathbb{N}$ such that the class of $\mathcal{H}$-free graphs is $c$-apex-perfect, but that no graph in $\mathcal{H}$ is a linear forest. Let $\ell$ be an odd integer with $\ell > \max_{H\in \mathcal{H}} |V(H)|$ and $\ell \geq 5$. Consider the graph $(c+1)C_{\ell}$. By the definition of $\ell$, it follows that every graph in $\mathcal{H}$ is $C_{\ell}$-free. Since each $C_{\ell}$-free induced subgraph of $(c+1)C_{\ell}$ is a linear forest, it follows that $(c+1)C_{\ell}$ is $\mathcal{H}$-free. 
    However, $(c+1)C_{\ell}$ is not $c$-apex-perfect, since deleting any $c$ vertices preserves at least one copy of $C_{\ell}$. 
    This contradicts the assumption that the class of $\mathcal{H}$-free graphs is $c$-apex-perfect. By this contradiction, we conclude that some graph in $\mathcal{H}$ is a linear forest. Similarly, by considering $\overline{(c+1)C_{\ell}}$ we conclude that some graph in $\mathcal{H}$ is the complement of a linear forest. 
\end{proof}

It turns out that for the case of forbidding just one graph $H$, the necessary conditions given by \cref{lem:LinForest} are sufficient. Namely, for a class of $H$-free graphs to be bounded-apex-perfect, $H$ must be both a linear forest and the complement of a linear forest. As shown in the next lemma, such graphs are precisely the induced subgraphs of $P_4$.

\begin{lemma} \label{lem:p4lin}
    If $H$ is both a linear forest and the complement of a linear forest, then $H$ is an induced subgraph of $P_4$. 
\end{lemma}
\begin{proof}
    By possibly replacing $H$ by $\overline{H}$, we may assume that $H$ is connected. Every component of a linear forest is a path, so $H$ is a path. If $|V(H)| \geq 5$, then its complement contains $K_3$ and so $H$ is not the complement of a linear forest. Hence $H$ is a path on at most $4$ vertices. Finally, the result follows from the fact $P_4$ is self-complementary, and hence any complement of an induced subgraph of $P_4$ is also an induced subgraph of $P_4$.
\end{proof}

This gives us the desired characterization for forbidding a single graph.

\begin{theorem}\label{thm:ExcludingOne}
        Let $H$ be a graph. The class of $H$-free graphs is bounded-apex-perfect if and only if $H$ is an induced subgraph of $P_4$.
\end{theorem}
\begin{proof}
    Let $H$ be a graph such that the class of $H$-free graphs is bounded-apex-perfect. By \cref{lem:LinForest}, $H$ is both a linear forest and the complement of a linear forest, which, by \cref{lem:p4lin}, implies that $H$ is an induced subgraph of $P_4$. 

    Conversely, note that $P_4$ is an induced subgraph of all odd holes and all odd antiholes. Hence, $P_4$-free graphs are perfect, proving the reverse implication.
\end{proof}

Over the next three sections we will consider the case of graph classes defined by two forbidden induced subgraphs. 



\section{Constructions of families that are not bounded-apex-perfect}
\label{sec:imperfect_families}
    
In this section, we consider three infinite families of graphs, shown in Figure \ref{fig:three_families}, that are not bounded-apex-perfect. The basic building blocks for the graphs in all three families are copies of $C_5$, since $C_5$ is the smallest obstruction to a graph being perfect. For each family, we prove that it is not bounded-apex-perfect and characterize the linear forests induced in its graphs and their complements. 

\begin{figure}
    \centering

    \begin{tikzpicture}[
    dot/.style={circle,fill=black,inner sep=1.8pt},
    every node/.style={font=\small}
]

\begin{scope}[shift={(0,-0.4)}]

    \begin{scope}[shift={(0.0,1.5)}]
        \coordinate (L11) at (0,0.65);
        \coordinate (L12) at (0.60,0.20);
        \coordinate (L13) at (0.35,-0.55);
        \coordinate (L14) at (-0.35,-0.55);
        \coordinate (L15) at (-0.60,0.20);
        \draw (L11)--(L12)--(L13)--(L14)--(L15)--cycle;
    \end{scope}

    \begin{scope}[shift={(2.2,1.5)}]
        \coordinate (L21) at (0,0.65);
        \coordinate (L22) at (0.60,0.20);
        \coordinate (L23) at (0.35,-0.55);
        \coordinate (L24) at (-0.35,-0.55);
        \coordinate (L25) at (-0.60,0.20);
        \draw (L21)--(L22)--(L23)--(L24)--(L25)--cycle;
    \end{scope}

    \begin{scope}[shift={(0.0,-0.8)}]
        \coordinate (L31) at (0,0.65);
        \coordinate (L32) at (0.60,0.20);
        \coordinate (L33) at (0.35,-0.55);
        \coordinate (L34) at (-0.35,-0.55);
        \coordinate (L35) at (-0.60,0.20);
        \draw (L31)--(L32)--(L33)--(L34)--(L35)--cycle;
    \end{scope}

    \begin{scope}[shift={(2.2,-0.8)}]
        \coordinate (L41) at (0,0.65);
        \coordinate (L42) at (0.60,0.20);
        \coordinate (L43) at (0.35,-0.55);
        \coordinate (L44) at (-0.35,-0.55);
        \coordinate (L45) at (-0.60,0.20);
        \draw (L41)--(L42)--(L43)--(L44)--(L45)--cycle;
    \end{scope}

    \foreach \v in {11,12,13,14,15,21,22,23,24,25,31,32,33,34,35,41,42,43,44,45}{
        \node[dot] at (L\v) {};
    }

    \node[align=center] at (1.1,-2.4) {Disjoint copies of $C_5$\\ of \cref{lem:c5s}};
\end{scope}

\begin{scope}[shift={(6.5,0)},scale=0.9]

    \coordinate (M1-1) at (-0.28,1.98);
    \coordinate (M1-2) at ( 0.28,1.98);
    \coordinate (M1-3) at ( 0.00,2.48);

    \coordinate (M2-1) at (1.82,0.48);
    \coordinate (M2-2) at (2.38,0.48);
    \coordinate (M2-3) at (2.10,0.98);

    \begin{scope}[shift={(1.30,-1.90)},rotate=-18]
        \coordinate (M3-2) at (-0.28,-0.1617);
        \coordinate (M3-1) at ( 0.28,-0.1617);
        \coordinate (M3-3) at ( 0.00, 0.3233);
    \end{scope}

    \begin{scope}[shift={(-1.30,-1.90)},rotate=18]
        \coordinate (M4-1) at (-0.28,-0.1617);
        \coordinate (M4-2) at ( 0.28,-0.1617);
        \coordinate (M4-3) at ( 0.00, 0.3233);
    \end{scope}
    
    \coordinate (M5-1) at (-2.38,0.48);
    \coordinate (M5-2) at (-1.82,0.48);
    \coordinate (M5-3) at (-2.10,0.98);

    \foreach \a/\b in {2/3,4/5}{
        \foreach \u in {1,2,3}{
            \foreach \v in {1,2,3}{
                \draw[gray] (M\a-\u)--(M\b-\v);
            }
        }
    }

    \foreach \a/\b in {1/2,3/4,5/1}{
        \foreach \v in {1,2,3}{
            \draw[gray] (M\a-\v)--(M\b-\v);
        }
    }

    \foreach \k in {1,2,3,4,5}{
        \draw (M\k-1)--(M\k-2)--(M\k-3)--cycle;
    }

    \foreach \k in {1,2,3,4,5}{
        \foreach \u in {1,2,3}{
            \node[dot] at (M\k-\u) {};
        }
    }

    \node[align=center] at (0,-3.1) {Modified clique-blowup of $C_5$\\ of \cref{lem:matchc5}};
\end{scope}

\begin{scope}[shift={(12,0)}, scale=0.9]

    \coordinate (R1-1) at (-0.28,1.98);
    \coordinate (R1-2) at ( 0.28,1.98);
    \coordinate (R1-3) at ( 0.00,2.48);

    \coordinate (R2-1) at (1.82,0.48);
    \coordinate (R2-2) at (2.38,0.48);
    \coordinate (R2-3) at (2.10,0.98);

    \begin{scope}[shift={(1.30,-1.90)},rotate=-18]
        \coordinate (R3-1) at (-0.28,-0.1617);
        \coordinate (R3-2) at ( 0.28,-0.1617);
        \coordinate (R3-3) at ( 0.00, 0.3233);
    \end{scope}

    \begin{scope}[shift={(-1.30,-1.90)},rotate=18]
        \coordinate (R4-1) at (-0.28,-0.1617);
        \coordinate (R4-2) at ( 0.28,-0.1617);
        \coordinate (R4-3) at ( 0.00, 0.3233);
    \end{scope}

    \coordinate (R5-1) at (-2.38,0.48);
    \coordinate (R5-2) at (-1.82,0.48);
    \coordinate (R5-3) at (-2.10,0.98);

    \foreach \a/\b in {1/2,2/3,3/4,4/5,5/1}{
        \foreach \u in {1,2,3}{
            \foreach \v in {1,2,3}{
                \draw[gray] (R\a-\u)--(R\b-\v);
            }
        }
    }

    \foreach \k in {1,2,3,4,5}{
        \foreach \u in {1,2,3}{
            \node[dot] at (R\k-\u) {};
        }
    }

    \node[align=center] at (0,-3.1) {Stable-set blowup of $C_5$\\ of \cref{lem:blowupC5}};
\end{scope}

\end{tikzpicture}
    
    \caption{Constructions of families that are not bounded-apex-perfect}
    \label{fig:three_families}
\end{figure}
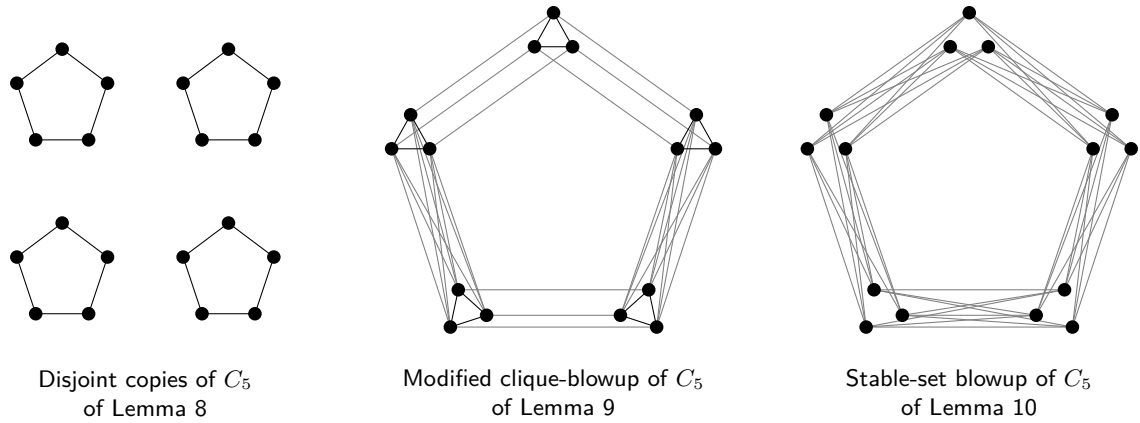

For describing the constructions the following terminology will be useful. For a graph $H$, the \emph{$H$-blowup} of a graph $G$ is obtained by replacing each vertex of $V(G)$ by a copy of $H$ (with different copies being pairwise disjoint) and making two such copies complete to one another whenever the two corresponding original vertices in $G$ are adjacent. When $H$ is a complete graph, we call it a \emph{clique-blowup}, and when $H$ is an edgeless graph, we call it a \emph{stable-set blowup}.

The first family consists of graphs that are disjoint unions of $n$ copies of $C_5$, which we showed in the proof of \cref{lem:LinForest} are not $(n-1)$-apex-perfect.

\begin{lemma} \label{lem:c5s}
    For $n\in \mathbb{N}$, let $F_n$ be the disjoint union of $n$ copies of $C_5$.

    Then the following hold:
    \begin{enumerate}[(i)]
        \item \label{claim:fig1i} The family $\mc{F} =\{F_n : n\in \mb{N} \}$ is not bounded-apex-perfect. 
        \item \label{claim:fig1ii} If $H$ is a linear forest contained in $F_n$ for some $n\in \mb{N}$, then $H$ is an induced subgraph of $rP_4$ for some $r \in \mathbb{N}$. 
        \item \label{claim:fig1iii} If $H$ is the complement of a linear forest and $H$ is an induced subgraph of $F_n$ for some $n\in \mb{N}$, then $H$ is an induced subgraph of $P_4$. 
    \end{enumerate}
\end{lemma}
\begin{proof}
    For any $c\in \mathbb{N}$, since $F_{c+1}$ contains $c+1$ disjoint induced copies of $C_5$, the deletion of at most $c$ vertices leaves behind at least one copy of $C_5$ as an induced subgraph. Hence $F_{c+1}$ is not $c$-apex-perfect, so we conclude that $\mathcal{F}$ is not $c$-apex-perfect for any $c \in \mathbb{N}$, proving \eqref{claim:fig1i}. 

    To prove \eqref{claim:fig1ii}, we note that each component of a linear forest $H$ contained in $F_n$ is induced in a component of $F_n$. The components of $F_n$ are copies of $C_5$, so each linear forest component is one of $P_4, P_3, K_2, K_2 + K_1, 2K_1$, or $K_1$. All are induced in $P_4$, so the linear forest $H$ is an induced subgraph of $rP_4$ for some $r \in \mathbb{N}$. 

    It remains to prove \eqref{claim:fig1iii}. Let $H$ be an induced subgraph of $F_n$ which is also the complement of a linear forest.  Since $C_5$ is not the complement of a linear forest, it follows that $H$ is an induced subgraph of $nP_4$. Now $H$ is both a linear forest and the complement of one, and hence $H$ is an induced subgraph of $P_4$ by Lemma \ref{lem:p4lin}. 
\end{proof}

The graphs in the second family consist of clique-blowups of $C_5$ where some of the bicliques corresponding to the edges of $C_5$ are replaced with perfect matchings. An example of one such modified clique blow-up appears in the center of \cref{fig:three_families}.

\begin{lemma} \label{lem:matchc5}
    For $n \in \mathbb{N}$, define a graph $F_n$ as follows. Let $V(F_n) := [5] \times [n]$ and
    \begin{itemize}
        \item For $i \in [5]$, we define $V_i := \{(i, j) : j \in [n]\}$; 
        \item For $i \in [5]$, the set $V_i$ is a clique;
        \item $V_1$ is complete to $V_5$ and $V_3$ is complete to $V_4$; 
        \item For all $j \in [n]$ and $(i, i') \in \{(1, 2), (2, 3), (4, 5)\}$, there exists an edge between $(i, j)$ and $(i', j)$. That is, between $V_1$ and $V_2$, between $V_2$ and $V_3$, and between $V_4$ and $V_5$, we add perfect matchings. 
    \end{itemize}

    Then the following hold: 
    \begin{enumerate}[(i)]
        \item \label{claim:fig2i} The family $\mathcal{F} := \{F_n : n \in \mathbb{N}\}$ is not bounded-apex-perfect. 
        \item \label{claim:fig2ii} If $H$ is a linear forest contained in $F_n$ for some $n \in \mathbb{N}$, then $H$ is an induced subgraph of one of $P_6$, $P_4 + K_2$ and $3K_2$. 
        \item \label{claim:fig2iii}  If $H$ is the complement of a linear forest and $H$ is an induced subgraph of $F_n$ for some $n \in \mathbb{N}$, then either $H$ is a complete graph, or $\overline{H}$ is an induced subgraph of $P_5$.
    \end{enumerate}
\end{lemma}

\begin{proof}
    For all $m \in [n]$, the vertex subset $V^m := \{(i, m) : i \in [5]\}$ of $F_n$ induces a copy of $C_5$. Hence, for any $c\in \mathbb{N}$, we have that $F_{c+1}$ contains $c+1$ disjoint induced copies of $C_5$ and is not $c$-apex-perfect, proving \eqref{claim:fig2i}. 

    To prove \eqref{claim:fig2ii}, we note that $V_1 \cup V_5$, $V_3 \cup V_4$ and $V_2$ is a partition of $V(F_n)$ into three cliques. Consequently, every induced subgraph of  $F_n$ has stability number at most 3. Every linear forest with stability number at most 3 is an induced subgraph of one of $P_6, P_4+K_2$, and $3K_2$, so \eqref{claim:fig2ii} follows. 

    It remains to prove \eqref{claim:fig2iii}. To this end, we first show that for all $n$, the graph $F_n$ is $K_4^-$-free. To see this, let $M$ be the set of edges defined by in the fourth bullet point (i.e. forming perfect matchings). Note that the edges in $M$ are not contained in a triangle in $F_n$, and thus not contained in any copies of $K_4^-$ in $F_n$. Furthermore, each component of $F_n - M$ (the graph obtained by removing the edges of $M$ from $F_n$) is a clique, and hence $K_4^-$-free. 

    It follows that if $H$ is an induced subgraph of $F_n$ for some $n$, and $\overline{H}$ is a linear forest, then $\overline{H}$ is $(K_2 + 2K_1)$-free. If $E(\overline{H}) = \emptyset$, then $H$ is a clique, as desired. Otherwise, let $e= xy$ be an edge of $\overline{H}$ where $x$ has degree 1 in $\overline{H}$. Since $\overline{H}$ is a linear forest, it follows that $y$ has degree at most 2 in $\overline{H}$. The common non-neighbours of $x$ and $y$ in $\overline{H}$ form a clique (of size at most two), since $\overline{H}$ is $(K_2 + 2K_1)$-free. Now $V(\overline{H}) \setminus \{x, y\}$ contains at most three vertices: one adjacent to $y$ (say $z$), and two non-adjacent to both $x$ and $y$ but adjacent to each other (say $u$ and $v$). If $z, u, v$ all exist, then $z$ is adjacent to $u$ or $v$ (for otherwise $\{x, z, u, v\}$ induces a copy of $(K_2 + 2K_1)$) and $\overline{H}$ is isomorphic to $P_5$. In every other case, $H$ has at most 4 vertices, and it is easy to verify that each  linear forest on at most 4 vertices is either an induced subgraph of $P_5$, or isomorphic to $K_2 + 2K_1$, or $|E(\overline{H})| = 0$, as claimed. 
\end{proof}

The graphs in the third family consist of stable-set blowups of $C_5$.

\begin{lemma} \label{lem:blowupC5}
    For $n \in \mathbb{N}$, define a graph $F_n$ as follows. Let $V(F_n) := [5] \times [n]$ and
    \begin{itemize}
        \item For $i \in [5]$, we define $V_i := \{(i, j) : j \in [n]\}$; 
        \item For each $i \in [5]$, $V_i$ is complete to $V_{(i+1)}$ (where we interpret $V_6$ as $V_1$).
    \end{itemize}
    That is, $F_n$ a stable-set blowup of $C_5$ into stable sets of size $n$.
    
    Then the following hold: 
    \begin{enumerate}[(i)]
        \item \label{claim:fig3i} The family $\mathcal{F} := \{F_n : n \in \mathbb{N}\}$ is not bounded-apex-perfect.  
        \item \label{claim:fig3ii} If $H$ is a linear forest contained in $F_n$ for some $n \in \mathbb{N}$, then $H$ is an induced subgraph of one of $P_4$ and $P_3 + rK_1$ for some $r \in \mathbb{N}$. 
        \item \label{claim:fig3iii} If $H$ is the complement of a linear forest  and $H$ is an induced subgraph of $F_n$ for some $n \in \mathbb{N}$, then $\overline{H}$ is an induced subgraph of $P_4$ or isomorphic to $2K_2$.
    \end{enumerate}
\end{lemma}
\begin{proof}
    For all $m \in [n]$, the vertex subset $V^m := \{(i, m) : i \in [5]\}$ of $F_n$ induces a copy of $C_5$. Hence, for any $c\in \mathbb{N}$, $F_{c+1}$ contains $c+1$ disjoint induced copies of $C_5$ and is not $c$-apex-perfect, proving \eqref{claim:fig3i}. 

    Let $H$ be a linear forest contained in $F_n$ for some $n \in \mathbb{N}$. 
    If $H$ contains at most one vertex from each $V_i$, but not from all (as that would induce a $C_5$), then $H$ is an induced subgraph of $P_4$. 
    If $H$ contains two distinct vertices $(i,m)$ and $(i, m')$ from the same $V_i$ for some $i \in [5]$, say $i = 2$, then $H$ does not also contain two vertices from $V_{1} \cup V_{3}$, as this would induce $C_4$. Hence vertices $(2, m)$ and $(2, m')$ have at most one neighbour in $H$. If they have exactly one neighbour, say in $V_{3}$, then $H$ contains no vertices from $V_{4}$, as this would create a vertex of degree greater than 2. Hence, $H$ consists of two vertices from $V_2$, one from $V_{3}$ and some vertices from $V_{5}$, which is $P_3 + rK_1$. On the other hand, if $(2,m)$ and $(2, m')$ have no neighbours in $H$, then the remaining vertices are drawn from $V_{2}, V_{4},$ and $V_{5}$. Thus, $H$ is the union of a stable set and a complete bipartite graph; and since $H$ is a linear forest, it follows that $H$ is an induced subgraph of $P_3 + rK_1$, proving \eqref{claim:fig3ii}. 

    It remains to prove \eqref{claim:fig3iii}. In the complement of $F_n$, each set $V_i$ induces a clique rather than a stable set. Since $C_5$ is self-complementary, so too is the overall structure of the sets $V_i$, such that the graph $\overline{F_n}$ is isomorphic to a clique-blowup of $C_5$. 

    Let $\overline{H}$ be a linear forest contained in $\overline{F_n}$ for some $n \in \mathbb{N}$.
    Again, if $H$ contains at most one vertex from each $V_i$, but not from all (as that would induce a $C_5$), then $H$ is an induced subgraph of $P_4$. 
    Additionally, since each $V_i$ is a clique, $H$ contains at most two vertices from $V_i \cup V_{i+2}$ (reading indices modulo $5$) for all $i \in [5]$. Thus, if $H$ contains two distinct vertices $(i, m)$ and $(i, m')$ from the same $V_i$ for some $i \in [5]$, say $V_1$, then all remaining vertices of $H$ come from $V_2 \cup V_5$. Together these two latter sets contribute at most two vertices to $H$, which are adjacent, such that $H$ is an induced subgraph of $2K_2$, and hence either an induced subgraph of $P_4$ or isomorphic to $2K_2$, as claimed. 
\end{proof}

\section{Families that are bounded-apex-perfect}
\label{sec:successful_pairs}

In this section, we prove for a number of pairs of graphs $H_1, H_2$ that the class of $\{H_1, H_2\}$-free graphs is bounded-apex-perfect. Later, we use these results to provide a full characterization of the bounded-apex-perfect classes defined by two forbidden induced subgraphs.

We will use the following theorem of Ramsey bounding the order of graphs with bounded size cliques and stable sets:
\begin{theorem}[Ramsey \cite{ramsey1987problem}]\label{thm:ramsey}
    For all integers $a, b \geq 1$ there exists an integer $R(a,b)$ such that any graph $G$ with $|V(G)| \geq R(a,b)$ either contains a clique of size $a$ or a stable set of size $b$.
\end{theorem}

We first show a simple application of Ramsey's Theorem: 

\begin{lemma}\label{em:rK_1_(rK_1)c-free}
    For each $r \geq 1$, the class of $\{rK_1, K_r\}$-free graphs is bounded-apex-perfect.
\end{lemma}
\begin{proof}
    Let $G$ be a $\{rK_1, K_r\}$-free graph. By \cref{thm:ramsey}, $|V(G)| < R(r,r)$. Hence, $G$ is $R(r,r)$-apex-perfect and thus the class is bounded-apex-perfect, as desired.
\end{proof}

Next, we will consider a more restricted setting. We still forbid a clique (in fact the smallest interesting clique, $K_3$), but instead of forbidding a stable set, we forbid $P_3 + rK_1$. 
We show that this class is not only bounded-apex-perfect, but in fact bounded-apex-bipartite.

\begin{lemma}\label{lem:K_3-P_3+rK_1-free-is-c-apex-perfect}
    For each $r \geq 0$, the class of $\{K_3, P_3 + rK_1\}$-free graphs is bounded-apex-perfect and bounded-apex-bipartite.
\end{lemma}
\begin{proof}
    Let $G$ be a $\{K_3, P_3 + rK_1\}$-free graph, and let $c := R(3, 4r -1)-5$. If $|V(G)| \leq c+4$, the graph $G$ is $c$-apex-bipartite as witnessed by deleting all but four vertices. Hence we may assume that $|V(G)| \geq c + 5 = R(3,4r-1)$. 

    Let $A \subseteq V(G)$ be a set of vertices inducing a maximum $P_3$-free subgraph in $G$, and let $B := V(G) \setminus A$. Since $G[A]$ is $P_3$-free, it is a disjoint union of cliques. Moreover, as $G$ and thus $G[A]$ is $K_3$-free, each component of $G[A]$ is either a single vertex or $K_2$. That is, $G[A]$ is the disjoint union of a matching $M$ and a stable set $S$. 

    \begin{claim}\label{claim:K_3-P_3+rK_1-claim-1}
        Every vertex $v \in B$ has at most $r-1$ non-neighbours in $S$.
    \end{claim}
    \begin{subproof}
        Suppose not, and let $v\in B$ be a vertex with $r$ non-neighbours in $S$. Let $s_1, \ldots, s_r \in S$ be $r$ such non-neighbours of $v$ in $S$.
        If $v$ has no neighbours in $A$, then the induced subgraph $G[A \cup \{v\}]$ is $P_3$-free, contradicting the maximality of $A$. Vertex $v$ thus has at least one neighbour $u \in A$. 
        
        Suppose that $u \in V(M)$, and let $u' \in V(M)$ be the vertex in $M$ such that $uu' \in E(M)$, Then, as $G$ is $K_3$-free, $\{v, u,u'\}$ induces a copy of $P_3$. The set $\{v, u, u', s_1, \ldots, s_r\}$ then induces a copy of $P_3 + rK_1$, contradicting the graph $G$ being $(P_3 + rK_1)$-free. 

        Hence, we have that $u \in S$. If $u$ is the sole neighbour of $v$ in $A$, then the induced subgraph $G[A \cup \{v\}]$ is $P_3$-free, again contradicting the maximality of $A$. Let $u' \in A \setminus \{u\}$ be a second neighbour of $v$ in $A$. Analogously as for $u$, it follows that $u' \in S$. Then $\{u,v,u'\}$ induces a copy of $P_3$, and thus $\{u, v, u', s_1, \ldots, s_r\}$ induces a copy of $P_3 + rK_1$. This again contradicts the graph $G$ being $(P_3 + rK_1)$-free.
    \end{subproof}

    \begin{claim}\label{claim:K_3_P_3+rK_1-claim-3}
        If $M$ contains $r+1$ edges, then the graph $G$ is bipartite.
    \end{claim}
    \begin{subproof}
        We show that $G$ is bipartite by showing that $G$ has maximum degree 1. To show this for vertices in $V(M)$, we first claim that $V(M)$ is anti-complete to $B$ in $G$. Suppose not, and let $v \in B$ be a vertex in $B$ with a neighbour $u$ in $M$. Because $G$ is $K_3$-free, $v$ is adjacent to at most one endpoint of each edge in $M$. Let $u' \in V(M)$ be the vertex such that $uu' \in E(M)$. Then $\{v,u,u'\}$ induces a copy of $P_3$ in $G$. Moreover, since $|E(M)| \geq r + 1$, there exist at least $r$ other edges in $M$ containing a non-neighbour of $v$. Taking a non-neighbour of $v$ in $r$ such edges, together with the vertices $\{v, u, u'\}$, induces a copy of $P_3 + rK_1$. This contradicts the graph $G$ being $(P_3 + rK_1)$-free. Hence, $V(M)$ is indeed anti-complete to $B$. In fact, since $M$ and $S$ are anti-complete by their definition, we have that $V(M)$ is anti-complete to $V(G) \setminus V(M)$. 

        Next, suppose that there exists a vertex $v \in V(G)$ of degree at least $2$. We first observe that since $V(M)$ is anti-complete to $V(G) \setminus V(M)$, and as $M$ itself is $1$-regular, $v \not \in V(M)$. Let $u, u' \in V(G)$ be two neighbours of $v$. Due to the facts that $v \not \in V(M)$ and $V(M)$ being anti-complete to $V(G) \setminus V(M)$, we have that $u, u' \in V(G) \setminus
         V(M)$ as well. Because $G$ is $K_3$-free, $\{v, u, u'\}$ induces $P_3$. Now taking one endpoint of $r$ edges in $M$ together with $\{v, u, u'\}$ gives a set inducing $P_3 + rK_1$, again a contradiction. 

         Thus, the maximum degree of $G$ is $1$, and hence each component of $G$ is either $K_1$ or $K_2$. Therefore, $G$ is bipartite, as claimed.
    \end{subproof}

    We now use these two claims to show that the set $B$ does not contain complex structures.

    \begin{claim}\label{claim:K_3-P_3+rK_1-claim-2}
        The set $B$ is stable.
    \end{claim}
    \begin{subproof}
        Suppose not, and let $uv \in E(G[B])$ be an edge witnessing that the set $B$ is not stable. Since $G$ is $K_3$-free, the vertices $u$ and $v$ have no common neighbours. By \cref{claim:K_3-P_3+rK_1-claim-1} $u$ and $v$ both have at most $r-1$ non-neighbours in $S$. We therefore observe that since each vertex in $S$ is non-adjacent to at least one of $u$ and $v$, the set $S$ contains at most $2r -2$ vertices. Hence, combining this with the bound from \cref{claim:K_3_P_3+rK_1-claim-3}, we get that:
        \[|A| = |V(M)| + |S| \leq 2r + (2r -2) = 4r -2.\]
        However, since $|V(G)| \geq R(3, 4r-1)$, and as $G$ is $K_3$-free, $G$ contains a stable set of size at least $4r-1$. Because such a stable set is $P_3$-free, its existence contradicts the maximality of the set $A$.
    \end{subproof}

    By \cref{claim:K_3-P_3+rK_1-claim-2} the set $B$ is stable, so the induced subgraph $G[S \cup B] \cong G - V(M)$ is bipartite.  By \cref{claim:K_3_P_3+rK_1-claim-3} we know that $|V(M)| \leq 2r \le c$. Hence, $G$ is $c$-apex-bipartite and thus $c$-apex-perfect, as desired.
\end{proof}

In the next result, both the clique and the stable set of \cref{em:rK_1_(rK_1)c-free} are replaced by slightly more complex forbidden induced subgraphs. We make use of the following 1988 result of Olariu.

\begin{theorem}[Olariu~\cite{olariu1988pawfree}] \label{thm:olariu}
Let $G$ be a connected $\overline{P_3+K_1}$-free graph. Then $G$ is either triangle-free or complete multipartite. 
\end{theorem}

\begin{lemma} \label{lem:paw}
For each $r \in \mathbb{N}$, the class of $\{P_3+rK_1, \overline{P_3+K_1}\}$-free graphs is bounded-apex-perfect.
\end{lemma}
\begin{proof} Let $c$ be as in \cref{lem:K_3-P_3+rK_1-free-is-c-apex-perfect}. We claim that 
$\{P_3+rK_1, \overline{P_3+K_1}\}$-free graphs are $c$-apex-perfect.

Let $G$ be a $\{P_3+rK_1, \overline{P_3+K_1}\}$-free graph. By \cref{thm:olariu}, every connected component of $G$ is either complete multipartite or triangle-free. Let $H$ be the union of the triangle-free connected components of $G$. By \cref{lem:K_3-P_3+rK_1-free-is-c-apex-perfect}, it follows that there is a set $S$ with $|S| \leq c$ such that $H - S$ is perfect. Since every connected component of $G - S$ is perfect (because it is either complete multipartite or an induced subgraph of $H - S$), it follows that $G - S$ is perfect, as desired. 
\end{proof}

It is possible to extract the following result from a result of Brandst{\"a}dt and Mahfud \cite{brandstadt2002maximum} (which only considers prime graphs), but we give a short, self-contained proof for completeness. 

\begin{lemma} \label{lem:diamond}
    The class of $\{K_4^-, \overline{K_4^-}\}$-free graphs is bounded-apex-perfect.
\end{lemma}

    \begin{proof}
        Let $G$ be a $\{K_4^-, \overline{K_4^-}\}$-free graph. In order to eventually invoke Ramsey's Theorem, we first address the case where $G$ has a relatively large stable set. 
        \begin{claim}\label{cl:stable}
            If $G$ contains a stable set of size at least $4$, then $G$ is either: 
            \begin{itemize}
                \item A graph obtained from a complete bipartite graph by deleting a (not necessarily perfect) matching; or
                \item A graph obtained from a complete bipartite graph by adding an isolated vertex. 
            \end{itemize}
        \end{claim}
        \begin{subproof}
        Let $S$ be a maximum stable set in $G$. It follows that every vertex in $V(G) \setminus S$ has a neighbour in $S$. Moreover, suppose that $v \in V(G) \setminus S$ has two distinct non-neighbours $u, u' \in S$. Let $v'$ be a neighbour of $v$ in $S$. Now $\{v, v', u, u'\}$ induces a copy of $\overline{K_4^-}$, a contradiction. Consequently, every vertex of $V(G) \setminus S$ has at most one non-neighbour in $S$. 

        Suppose for a contradiction that $G - S$ contains two adjacent vertices $u, v$. Since $|S| \geq 4$ and each of $u, v$ has at most one non-neighbour in $S$, it follows that there are two distinct vertices $x, y \in S$ such that $u, v \in N(x) \cap N(y)$. Now, $\{u, v, x, y\}$ induces a copy of $K_4^-$, a contradiction. 

        It follows that $T := V(G) \setminus S$ is a stable set in $G$. As above, no vertex in $S$ has a neighbour and two non-neighbours in $T$. Let $S'$ be the set of vertices in $S$ with no neighbour in $T$. It follows that each vertex in $S \setminus S'$ has at most one non-neighbour in $T$. 

        Suppose first that $S' = \emptyset$. Then, as each vertex in $T$ has at most one non-neighbour in $S$ and vice versa, it follows that the first outcome of the claim holds. 

        Therefore, we may assume that $S' \neq \emptyset$. Note that each vertex in $S'$ is isolated in $G$. If $G$ is edgeless, then the first outcome of the claim holds (one side of the bipartition is empty). If $G$ contains an edge as well as at least two isolated vertices, then $G$ contains a copy of $\overline{K_4^-}$, a contradiction. Therefore, we may assume that $|S'| = 1$. The graph $G' := G[(S \setminus S') \cup T]$ is bipartite, and if $G'$ is complete bipartite, then the second outcome of the claim holds. 
        If $G'$ is not complete bipartite, then it contains an induced copy of $K_2 + K_1$. Along with the isolated vertex in $S'$, this forms an induced copy of $\overline{K_4^-}$, a contradiction. 
        \end{subproof}
    By \cref{cl:stable} applied to both $G$ and $\overline{G}$, and noticing that the outcomes of the claim are bipartite and hence perfect, we may assume that $G$ has no clique and no stable set of size at least 4, and so by \cref{thm:ramsey}, we have $|V(G)| \leq R(4, 4).$ 
    \end{proof}

\section{Proof of \cref{thm:main}}
\label{sec:main_proof}

We first include an auxiliary result we will use to prove \cref{thm:main}. A \emph{pseudo-split graph} is a graph $G$ with vertex set $K \cup S \cup R$ where $K$ is a clique, $S$ is a stable set, $R$ is either empty or induces a copy of $C_5$, and $R$ is complete to $K$ and anti-complete to $S$. 
In 1994, Maffray and Preissmann gave the following forbidden induced subgraph characterization of pseudo-split graphs.
\begin{theorem}[Maffray and Preissmann \cite{maffray1994linear}]\label{thm:PseudoSplit}
    The class of pseudo-split graphs is exactly the class of $\{2K_2, C_4\}$-free graphs.
\end{theorem}
 We also use the following result on pseudo-split graphs.
\begin{theorem}[Blázsik, Hujter, Pluhár, \& Tuza, Corollary 1.2~\cite{BLAZSIK199351}]\label{thm:2K2-C4-free-structure}
    Let $G$ be a $\{2K_2, C_4\}$-free graph. Then $G$ is either a split graph, or there are exactly five distinct vertices $v_1, \ldots, v_5$ such that for each $i \in [5]$ the graph $G- v_i$ is a split graph.
\end{theorem}

We are now ready to prove \cref{thm:main}, which we restate for the reader's convenience. 
\main*

\begin{proof}
    We first show that if the pair $\{H_1, H_2\}$ satisfies one of the conditions of \cref{thm:main}, then the class of $\{H_1, H_2\}$-free graphs is bounded-apex-perfect.  Note that by \cref{lem:complement}, for any $c \in \mathbb{N}$ we have that $\{H_1, H_2\}$-free graphs are $c$-apex-perfect if and only if $\{\overline{H_1}, \overline{H_2}\}$-free graphs are $c$-apex-perfect. Therefore, it suffices to consider the cases stated in the theorem, with those obtained by taking complements of both $H_1$ and $H_2$ subsequently following. 
    \begin{enumerate}
        \item If $H_1$ is an induced subgraph of $P_4$, then the result follows from \cref{thm:ExcludingOne}. 
        \item If $H_1$ is an induced subgraph of $\overline{P_3+K_1}$ and $H_2$ is an induced subgraph of $P_3+rK_1$ for some $r \in \mathbb{N}$, then the result follows from \cref{lem:paw}.  
        \item If $H_1$ is a complete graph and $H_2$ is edgeless, then the result follows from \cref{em:rK_1_(rK_1)c-free}. 
        \item If $H_1$ is an induced subgraph of $\overline{K_2+K_1}$ and $H_2$ is an induced subgraph $K_2 + 2K_1$, then the result follows from \cref{lem:diamond}.
        \item If $H_1 \cong C_4$ and $H_2 \cong 2K_2$, then by \cref{thm:2K2-C4-free-structure}, it follows that $\{H_1, H_2\}$-free graphs are $1$-apex-perfect.
    \end{enumerate}

    It remains to show that if the class of $\{H_1, H_2\}$-free graphs is bounded-apex-perfect, then one of the above cases holds. From \cref{lem:LinForest}, it follows that one of $H_1$ and $H_2$ is a linear forest; and one of $H_1$ and $H_2$ is the complement of a linear forest. By \cref{lem:p4lin}, it follows that if one of $H_1$ and $H_2$ is both a linear forest and the complement of a linear forest, then Case~\ref{case:p4} applies. Therefore, from now on, we assume that $H_1$ is the complement of a linear forest, and $H_2$ is a linear forest, and that neither $H_1$ nor $H_2$ is an induced subgraph of $P_4$. From \cref{lem:c5s} (and considering the family given by the complements of the graphs in \cref{lem:c5s}), it follows that both $\overline{H_1}$ and $H_2$ are induced subgraphs of $rP_4$ for some $r \in \mathbb{N}$.  

 By \cref{lem:blowupC5}, one of the following holds: 
    \begin{itemize}
        \item $H_2$ is an induced subgraph $P_4$ or $P_3+rK_1$; or
        \item $\overline{H_1}$ is an induced subgraph of $P_4$ or $2K_2$. . 
    \end{itemize}
    Since neither of $\overline{H_1}$ and $H_2$ is an induced subgraph of $P_4$, we conclude that one of the following holds: 
    \begin{itemize}
        \item $H_2$ is an induced subgraph of $P_3+rK_1$; or
        \item $\overline{H_1} \cong 2K_2$. 
    \end{itemize}
    If $\overline{H_1} \cong 2K_2$, then by considering the pair $\{\overline{H_1}, \overline{H_2}\}$, it follows from \cref{lem:blowupC5} that Case~\ref{case:c42p2} applies. Hence, we may assume that $H_2$ is an induced subgraph of $P_3+rK_1$. Similarly, by again considering \cref{lem:blowupC5} for the pair $\{\overline{H_1}, \overline{H_2}\}$, we conclude that $\overline{H_1}$ is an induced subgraph of $P_3+rK_1$. 
    
    By \cref{lem:matchc5}, one of the following holds: 
    \begin{itemize}
        \item $H_2$ is an induced subgraph of $P_6$, $P_4+K_2$, or $3K_2$; or
        \item $\overline{H_1}$ is an induced subgraph of $P_5$, or $H_1$ is a complete graph. 
    \end{itemize}
    Since neither of $H_1, H_2$ is an induced subgraph of $P_4$, and since $H_2$ is an induced subgraph of $P_3+rK_1$, we conclude that one of the following holds: 
    \begin{itemize}
        \item $H_2$ is an induced subgraph of $P_3+K_1$, or $K_2 + 2K_1$; or
        \item $\overline{H_1}$ is an induced subgraph of $P_3+K_1$, or $H_1$ is a complete graph. 
    \end{itemize}
    Taking the complements of the graphs constructed in \cref{lem:matchc5}, we similarly conclude that: 
    \begin{itemize}
        \item $\overline{H_1}$ is an induced subgraph of $P_3+K_1$, or $K_2 + 2K_1$; or
        \item $H_2$ is an induced subgraph of $P_3+K_1$, or $\overline{H_2}$ is a complete graph. 
    \end{itemize}
    Suppose first that $H_1$ is a complete graph. If $\overline{H_2}$ is also a complete graph, then Case~\ref{case:ramsey} applies. Now assume that $\overline{H_2}$ is not a complete graph. If $H_2$ is an induced subgraph of $P_3+K_1$, then Case~\ref{case:triangle} applies; therefore, we assume that $H_2$ is not an induced subgraph of $P_3+K_1$. It follows that $\overline{H_1}$ is an induced subgraph of $P_3+K_1$ or $K_2+2K_1$. Since $H_1$ is a complete graph, it follows that $H_1$ is an induced subgraph of $K_3$ and Case~\ref{case:triangle} applies. 

    Now we assume that neither $H_1$ nor $\overline{H_2}$ is a complete graph. Then, by our previous logic, both $\overline{H_1}$ and $H_2$ are induced subgraphs of $P_3 + rK_1$ for some $r \in \mathbb{N}$. If $\overline{H_1}$ is an induced subgraph of $P_3 + K_1$, it follows that Case~\ref{case:triangle} applies. Hence, as $\overline{H_1}$ is also assumed to not be complete, we find that $H_2$ is an induced subgraph of $P_3 +K_1$ or of $K_2+2K_1$. If $H_2$ is not an induced subgraph of $P_3 + K_1$, then it follows that $\overline{H_1}$ is an induced subgraph of $P_3 + K_1$ or of $K_2 + 2K_1$. Therefore, Case~\ref{case:triangle} and Case~\ref{case:p3+p1} apply respectively. Finally, if $H_2$ is an induced subgraph of $P_3 + K_1$, taking the complement of both $H_1$ and $H_2$ suffices for Case~\ref{case:triangle} to apply. The result follows.
\end{proof}

In the remainder of the paper, we prove results analogous to \cref{thm:main} for several important subclasses of perfect graphs. Observe that as we consider subclasses of perfect graphs, the sets $\mathcal{H}$ of forbidden induced subgraphs giving a positive answer must be contained in one of the cases of \cref{thm:main}. In some of the upcoming proofs we will consider these cases explicitly, whereas in others we choose to give a self-contained proof to facilitate a better understanding of the structures giving rise to the restrictions on the forbidden induced subgraphs.


\section{Bounded-apex-chordal graph classes} \label{sec:chordal}

A \emph{chordal graph} is a graph containing no hole, where a \emph{hole} is  an induced cycle on four or more vertices. In this section, we answer \cref{question:mainquestion} for chordal graphs, namely characterizing for which sets of graphs $\mathcal{H}$ with $|\mathcal{H}| \leq 2$ the class of $\mathcal{H}$-free graphs is bounded-apex-chordal. We begin by considering classes defined by a single forbidden induced subgraph $H$.

\begin{theorem}\label{thm:apex-chordal_characterization_single_isg}
    Let $H$ be a graph. The class of $H$-free graphs is bounded-apex-chordal if and only if $H$ is an induced subgraph of $P_3$.
\end{theorem}
\begin{proof}  
    We first observe that if $H$ is an induced subgraph of $P_3$, the class of $H$-free graphs is chordal. Hence it suffices to show the forward direction of the statement.

    Let $c \in \mathbb{N}$ be given such that the class of $H$-free graphs is $c$-apex-chordal. We observe that the class of $H$-free graphs, therefore, does not contain $(c+1)C_4$ and $(c+1)C_5$. Thus, $H$ is an induced subgraph of both these graphs. Since bounded-apex-chordal graphs are bounded-apex-perfect, $H$ is a linear forest by \cref{lem:LinForest}. Thus, each component of $H$ is an induced subgraph of $P_3$. 
    Now consider the join of $c+1$ copies of $C_4$, that is, the graph consisting of $c+1$ copies of $C_4$, pairwise complete to one another. Since this construction is not $c$-apex-chordal, it contains $H$. If $H$ has at least two components, they are both contained in a single copy of $C_4$, and hence $H \cong 2K_1$, an induced subgraph of $P_3$. The statement follows.
\end{proof}

Now we turn to classes defined by two forbidden induced subgraphs. For the following proof, we use the following notion: a graph $G$ is said to be a \emph{co-matching} if it is contained in the complement of a matching. That is, there exists an $r \in \mathbb{N}$ such that $G$ is an induced subgraph of $\overline{rK_2}$. Note that the class of co-matchings is hereditary. We make use of the following two facts about co-matchings:

\begin{lemma}\label{lem:triangle-free-co-matching}
    If $G$ is a $K_3$-free co-matching, then $G$ is an induced subgraph of $C_4$.
\end{lemma}
\begin{proof}
    Let $r \in \mathbb{N}$ be given such that $G$ is an induced subgraph of $\overline{rK_2}$. Observe that taking one endpoint of three distinct non-edges in the anti-matching induces a copy of $K_3$. Hence, $G$ contains endpoints of at most two non-edges in the anti-matching. Two such non-edges together induce a copy of $C_4$. Hence, $G$ is an induced subgraph of $C_4$, as desired.
\end{proof}

\begin{lemma}\label{lem:diamond-free-co-matching}
    If $G$ is a $K_4^-$-free co-matching containing a copy of $K_3$, then $G$ is a complete graph.
\end{lemma}
\begin{proof}
    Suppose for the sake of contradiction that $G$ contains a non-edge with endpoints $x$ and $y$. Let $T \subset V(G)$ induce a copy of $K_3$. Then $T$ contains at most one of $x$ and $y$, and thus contains two distinct vertices $u,v \in V(G)\setminus \{x,y\}$. Since $G$ is a co-matching, the only non-edge in $G[\{x,y,u,v\}]$ has endpoints $x$ and $y$, that is, the set $\{x,y,u,v\}$ induces a copy of $K_4^-$, a contradiction.
\end{proof}

We are now ready to proceed to the case of forbidding two induced subgraphs.

\begin{theorem}
\label{thm:apex-chordal_characterization}
    Let $H_1$ and $H_2$ be graphs. The class of $\{H_1, H_2\}$-free graphs is bounded-apex-chordal if and only if -- up to switching the labels of $H_1$ and $H_2$ -- one of the following holds:
    \begin{enumerate}
        \item $H_1$ is an induced subgraph of $P_3$;
        \item $H_1$ is an induced subgraph of $P_4$ and $H_2 \cong C_4$, in which case $c=0$;
        \item $H_1 \cong 2K_2$ and $H_2 \cong C_4$; or
        \item $H_1$ is a complete graph and $H_2$ is edgeless.
    \end{enumerate}
\end{theorem}
\begin{proof}
    We first show that if $H_1, H_2$ satisfy one of the cases, then $\{H_1, H_2\}$-free graphs are bounded-apex-chordal. 
    \begin{enumerate}
        \item Follows from \cref{thm:apex-chordal_characterization_single_isg}.
        \item If $H_1$ is an induced subgraph of $P_4$ and $H_2 \cong C_4$, then each $\{H_1, H_2\}$-free graph is chordal, since any hole is either $C_4$ or contains an induced $P_4$.
        \item By \cref{thm:PseudoSplit}, the class of $\{2K_2, C_4\}$-free graphs is exactly the class of pseudo-split graphs. Since split graphs are chordal, by \cref{thm:2K2-C4-free-structure}, the class of pseudo-split graphs is $1$-apex-chordal.
        \item Once again, this is immediate from Ramsey's Theorem, \cref{thm:ramsey}.
    \end{enumerate}

    For the forward direction, suppose that $c \in \mathbb{N}$ is given such that the class of $\{H_1, H_2\}$-free graphs is $c$-apex-chordal. 

    \begin{claim}
        If the class of $\{H_1, H_2\}$-free graphs is $c$-apex-chordal, at least one of $H_1$ and $H_2$ is a linear forest.
    \end{claim}
    \begin{subproof}
        Suppose that neither $H_1$ nor $H_2$ is a linear forest. That is, $H_1$ and $H_2$ each contain either a cycle or a vertex of degree at least $3$. Let $l > |V(H_1)|, |V(H_2)|$ be an integer. We observe that the graph $(c+1)C_l$ contains neither $H_1$ nor $H_2$ as it contains no vertices of degree at least $3$ and contains no cycles of size at most $\max(|V(H_1)|, |V(H_2)|)$. However, it is also not $c$-apex-chordal, a contradiction. Hence, at least one of the graphs $H_1$ and $H_2$ is a linear forest, as desired.
    \end{subproof}
    
     Without loss of generality, let $H_1$ be a linear forest.
     The case where $H_1$ or $H_2$ is an induced subgraph of $P_3$ corresponds to Case~1 of \cref{thm:apex-chordal_characterization} (up to switching the labels of $H_1$ and $H_2$). Hence, we assume that neither $H_1$ nor $H_2$ is an induced subgraph of $P_3$. 

     \begin{claim}\label{claim:bounded-apex-chordal-linear-forest-not-P3-implies-co-matching}
         If the class of $\{H_1, H_2\}$-free graphs is $c$-apex-chordal and $H_1$ is a linear forest but not an induced subgraph of $P_3$, then $H_2$ is a co-matching.
     \end{claim}
     \begin{subproof}
        We consider the graph obtained by taking $c+1$ copies of $C_4$ and making them complete to one another. We note that the only linear forests contained in this construction are induced subgraphs of $P_3$ and hence this construction does not contain $H_1$. Moreover, as it contains $c+1$ disjoint copies of $C_4$, it is not $c$-apex-chordal either and thus contains $H_2$. The claim follows from the fact that the construction is itself a co-matching and co-matchings are hereditary.
     \end{subproof}

     Hence, by \cref{claim:bounded-apex-chordal-linear-forest-not-P3-implies-co-matching}, we assume that $H_2$ is a co-matching. 
     First, suppose that $H_2$ is $K_3$-free. Then by \cref{lem:triangle-free-co-matching}, $H_2$ is an induced subgraph of $C_4$. 
     Also, since $H_2$ is not an induced subgraph of $P_3$, we obtain that $H_2 \cong C_4$. Consider the $K_{c+1}$-blowup of $C_5$. This graph is not $c$-apex-chordal by virtue of containing $c+1$ disjoint copies of $C_5$. Since this graph is $C_4$-free, it contains $H_1$. The only linear forests contained in this graph are induced subgraphs of $P_4$ and $2K_2$. Hence, either $H_1$ is an induced subgraph of $P_4$, which corresponds to Case~2, or $H_1 \cong 2K_2$, which corresponds to Case~3 of \cref{thm:apex-chordal_characterization}. Thus, we assume that $H_2$ contains a copy of $K_3$.  
     
     Consider the graph $K_{c+2, c+2}$, which is not $c$-apex-chordal and thus contains $H_1$ or $H_2$. 
     Because it is bipartite, it is $K_3$-free and therefore contains $H_1$. The only linear forests contained in $K_{c+2,c+2}$ are induced subgraphs of $P_3$ and stable sets. Hence, we may assume that $H_1 \cong rK_1$ for some $r \geq 3$. 
     
     Finally, consider the graph consisting of four disjoint cliques $K_{(1)}, K_{(2)},K_{(3)},K_{(4)}$ of size $c+1$, where $K_{(1)}$ is complete to $K_{(2)}$ and $K_{(3)}$ is complete to $K_{(4)}$, and there is a perfect matching between $K_{(2)}$ and $K_{(3)}$ and a perfect matching between $K_{(1)}$ and $K_{(4)}$. This graph is not $c$-apex-chordal, by virtue of containing $c+1$ disjoint copies of $C_4$. Moreover, it has stability number $2$ and hence does not contain $H_1$. It therefore contains $H_2$. Additionally, it is $K_4^-$-free, implying that $H_2$ is also $K_4^-$-free. 
     Then, by \cref{lem:diamond-free-co-matching}, we conclude that $H_2$ is a complete graph, which corresponds to Case 4 of \cref{thm:apex-chordal_characterization} after relabeling $H_1$ and $H_2$.
\end{proof}






\section{Bounded-apex-interval graph classes}\label{sec:interval}

In this section, we answer \cref{question:mainquestion} for interval graphs. An \emph{interval graph} is the intersection graph of a finite set of intervals of $\mathbb{R}$. 
In 1962, Lekkerkerker and Boland~\cite{Lekkerkerker1962representation} characterized the class of interval graphs via a family of forbidden induced subgraphs, shown in \cref{fig:forbidden_isg_interval_graphs}. They also produced another characterization of interval graphs using the following definition. An \emph{asteroidal triple} is a set of three pairwise non-adjacent vertices such that each pair is joined by a path disjoint from the third vertex and its neighbourhood. A graph is said to be \emph{AT-free} if it does not contain an asteroidal triple. Lekkerkerker and Bolan~\cite{Lekkerkerker1962representation} proved that interval graphs are precisely the graphs that are both chordal and AT-free. We will make use of our characterizations of bounded-apex-chordal classes from \cref{sec:chordal}.

\begin{figure}[h!]\label{fig:interval}
    \centering

\begin{tikzpicture}[
    dot/.style={circle,fill=black,inner sep=1.8pt},
    every node/.style={font=\small}
]

\begin{scope}[shift={(0,0)}]
    \coordinate (a) at (0,1.3);
    \coordinate (b) at (1.6,1.3);
    \coordinate (c) at (1.6,0);
    \coordinate (d) at (0,0);

    \draw (a)--(b)--(c)--(d)--(a);

    \node[dot,label=above:{$1$}] at (a) {};
    \node[dot,label=above:{$2$}] at (b) {};
    \node[dot,label=below:{$\dots$}] at (c) {};
    \node[dot,label=below:{$n$}] at (d) {};

    \node at (0.8,-0.8) {$\left(n>3\right)$};
\end{scope}

\begin{scope}[shift={(5.5,0.1)}]
    \node[dot] (t2) at (0,1.8) {};
    \node[dot] (t1) at (0,0.9) {};
    \node[dot] (r)  at (0,0)   {};

    \node[dot] (l1) at (-0.7,-0.7) {};
    \node[dot] (l2) at (-1.4,-1.4) {};
    \node[dot] (m1) at (0.7,-0.7) {};
    \node[dot] (m2) at (1.4,-1.4) {};

    \draw (t2)--(t1)--(r);
    \draw (r)--(l1)--(l2);
    \draw (r)--(m1)--(m2);

\end{scope}

\begin{scope}[shift={(9.2,0.15)}]
    \node[dot] (b1) at (0,0) {};
    \node[dot] (b2) at (0.9,0) {};
    \node[dot] (b3) at (1.8,0) {};
    \node[dot] (b4) at (2.7,0) {};
    \node[dot] (b5) at (3.6,0) {};

    \node[dot] (u)  at (1.8,0.95) {};
    \node[dot] (p)  at (1.8,-0.95) {};

    \draw (b1)--(b2)--(b3)--(b4)--(b5);

    \draw (u)--(b1);
    \draw (u)--(b2);
    \draw (u)--(b3);
    \draw (u)--(b4);
    \draw (u)--(b5);

    \draw (u)--(p);
\end{scope}

\begin{scope}[shift={(0,-4.0)}]
    \node[dot] (b1) at (0,0) {};
    \node[dot,label=below:{$1$}] (b2) at (0.9,0) {};
    \node[dot,label=below:{$2$}] (b3) at (1.8,0) {};
    \node[dot] (b4) at (2.7,0) {};
    \node[dot] (b5) at (3.6,0) {};
    \node[dot,label=below:{$n$}] (b6) at (4.5,0) {};
    \node[dot] (b7) at (5.4,0) {};

    \node[dot] (u) at (2.7,1.05) {};
    \node[dot] (t) at (2.7,1.95) {};

    \draw (b1)--(b2)--(b3)--(b4)--(b5)--(b6)--(b7);

    \draw (u)--(t);

    \foreach \x in {b2,b3,b4,b5,b6}
        \draw (u)--(\x);

    \node at (2.7,-0.35) {$\cdots$};
    \node at (2.7,-1.0) {$\left(n>1\right)$};
\end{scope}

\begin{scope}[shift={(8.2,-4.0)}]
    \node[dot] (b1) at (0,0) {};
    \node[dot,label=below:{$1$}] (b2) at (0.9,0) {};
    \node[dot,label=below:{$2$}] (b3) at (1.8,0) {};
    \node[dot] (b4) at (2.7,0) {};
    \node[dot] (b5) at (3.6,0) {};
    \node[dot,label=below:{$n$}] (b6) at (4.5,0) {};
    \node[dot] (b7) at (5.4,0) {};

    \node[dot] (u1) at (1.9,1.05) {};
    \node[dot] (u2) at (3.5,1.05) {};
    \node[dot] (a)  at (2.7,2.0)  {};

    \draw (b1)--(b2)--(b3)--(b4)--(b5)--(b6)--(b7);

    \draw (a)--(u1)--(u2)--(a);

    \draw (u1)--(u2);

    \foreach \x in {b1,b2,b3,b4,b5,b6,b7}{
        \draw (u1)--(\x);
        \draw (u2)--(\x);
    }

    \node at (2.7,-0.35) {$\cdots$};
    \node at (2.7,-1.0) {$\left(n>0\right)$};
\end{scope}

\end{tikzpicture}

    \caption{Forbidden induced subgraphs of interval graphs (see \cite{Lekkerkerker1962representation}).} 
    \label{fig:forbidden_isg_interval_graphs}
\end{figure}
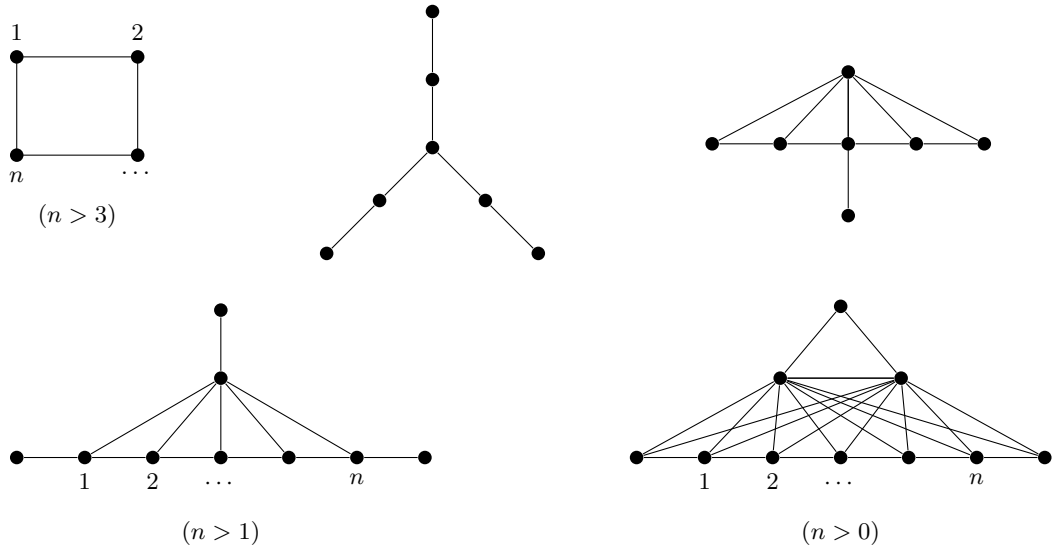

Our first result concerns the case where we exclude a single graph.
\begin{theorem}\label{thm:apex-interval_characterization_single_isg}
    Let $H$ be a graph. The class of $H$-free graphs is bounded-apex-interval if and only if $H$ is an induced subgraph of $P_3$.
\end{theorem}
\begin{proof}
    First, as each of the forbidden induced subgraphs of interval graphs contains $P_3$ (see \cref{fig:forbidden_isg_interval_graphs}), $P_3$-free graphs are interval graphs.
    Conversely, by \cref{thm:apex-chordal_characterization_single_isg}, if $H$ is not an induced subgraph of $P_3$, then the class of $H$-free graphs is not bounded-apex-chordal, and hence not bounded-apex-interval.
\end{proof}

Our second result for bounded-apex-interval graphs concerns the case where we exclude two forbidden induced subgraphs.

\begin{theorem}\label{thm:apex-interval_characterization}
    Let $H_1, H_2$ be graphs. The class of $\{H_1, H_2\}$-free graphs is bounded-apex-interval if and only if -- up to switching the labels of $H_1$ and $H_2$ -- one of the following holds:
    \begin{enumerate}
        \item $H_1$ is an induced subgraph of $P_3$;
        \item $H_1$ is an induced subgraph of $P_4$ and $H_2 \cong C_4$, in which case $c=0$; or
        \item $H_1$ is a complete graph and $H_2$ is edgeless.
    \end{enumerate}
\end{theorem}
\begin{proof}
    We first show that if $H_1, H_2$ satisfy one of the cases, then $\{H_1, H_2\}$-free graphs are bounded-apex-interval.
    \begin{enumerate}
        \item Follows from \cref{thm:apex-interval_characterization_single_isg}. 
        \item Observe that by excluding $P_4$ and $C_4$, we exclude all forbidden induced subgraphs of interval graphs (see \cref{fig:forbidden_isg_interval_graphs}). Hence, the class of $\{P_4, C_4\}$-free graphs, and any of its subclasses, are a subclass of interval graphs and thus bounded-apex-interval for $c=0$.
        \item Once again, this is immediate from Ramsey's \cref{thm:ramsey}.
    \end{enumerate}
        
    Conversely, suppose that the class of $\{H_1, H_2\}$-free graphs is $c$-apex-interval for some $c \in \mathbb{N}$. 
    Since interval graphs are a subclass of chordal graphs, it suffices to consider the four cases of \cref{thm:apex-chordal_characterization}:
    \begin{enumerate}
        \item The case where $H_1$ is an induced subgraph of $P_3$ corresponds to Case 1 of \cref{thm:apex-interval_characterization}.
        \item The case where $H_1$ is an induced subgraph of $P_4$ and $H_2 \cong C_4$ corresponds to Case 2 of \cref{thm:apex-interval_characterization}. 
        \item Consider the case where $H_1 \cong 2K_2$ and $H_2 \cong C_4$. 
        For $c \in \mathbb{N}$, let $G_c$ with $V(G_c) = K\cup S$ be the complete split graph (that is, every vertex in $S$ is adjacent to every vertex of $K$) with $|K| = |S| = c+3$. We claim that the graph resulting from removing a perfect matching between $K$ and $S$ is not $c$-apex-interval. 

        Observe that any three distinct $v_1, v_2, v_3$ in the stable set form an asteroidal triple: they form a stable set, and for any two $v_i, v_j$, with $v_k$ being the third vertex, there is a path from $v_i$ to $v_j$ going through the unique non-neighbour of $v_k$ in the clique. The deletion of at most $c$ vertices leaves at least three vertices in the stable set with their corresponding unique non-neighbour in the clique, and thus the resulting graph is not AT-free, and in particular not an interval graph. Hence the considered complete split graph is not $c$-apex-interval, as desired.

        Note that the class of $\{H_1,H_2\}$-free graphs contains this construction and is thus not bounded-apex-interval.
        
        \item The case where $H_1$ is a complete graph and $H_2$ is edgeless corresponds to Case 4 of \cref{thm:apex-interval_characterization}.
    \end{enumerate}    
    This completes the forward direction of \cref{thm:apex-interval_characterization}. The result follows.
\end{proof}

\section{Bounded-apex-split graph classes}\label{sec:split}

In this section, we answer \cref{question:mainquestion} for split graphs. In 1977, Foldes and Hammer~\cite{split-graph-forbidden-induced-subgraphs} characterized split graphs as exactly the class of $\{C_4, C_5, 2K_2\}$-free graphs. Note that the class of split graphs is thus self-complementary.

\begin{theorem}\label{thm:apex-split-H-free}
    Let $H$ be a graph. The class of $H$-free graphs is bounded-apex-split if and only if $H$ is an induced subgraph of $K_2$ or of $2K_1$.
\end{theorem}
\begin{proof}
    First, if $H$ is an induced subgraph of $K_2$, every $H$-free graph is stable and hence split. Similarly, if $H$ is an induced subgraph of $2K_1$, every $H$-free graph is complete and thus split.
    
    Now assume the class of $H$-free graphs is bounded-apex-split. Since bounded-apex-split graphs are also bounded-apex-perfect, by \cref{thm:ExcludingOne}, $H$ is an induced subgraph of $P_4$ for the class of $H$-free graphs to be bounded-apex-split. Next, observe that for any $c \in \mathbb{N}$ the graph $(c+2)K_2$ is not $c$-apex-split. Thus, $H$ is an induced subgraph of $rK_2$ for some $r \in \mathbb{N}$. Analogously, the complementary construction given by $\overline{(c+2)K_2}$ for $c \in \mathbb{N}$ (that is, graphs resulting from removing a perfect matching from a complete graph) is not $c$-apex-split. The only graphs that are induced subgraphs of $P_4$, $rK_2$, and $\overline{rK_2}$ for some $r \in \mathbb{N}$ are induced subgraphs of $K_2$ or of $2K_1$. The result follows.
\end{proof}

Next, we characterize the hereditary bounded-apex-split graph classes defined by two forbidden induced subgraphs.

\begin{theorem}
    \label{thm:apexsplit}
    Let $H_1$ and $H_2$ be graphs. The class of $\{H_1, H_2\}$-free graphs is $c$-apex-split if and only if -- up to switching the labels of $H_1$ and $H_2$ -- one of the following holds: 
    \begin{enumerate}
        \item $H_1$ is an induced subgraph of $K_2$ or of $2K_1$, in which case $c = 0$;
        \item $H_1$ is an induced subgraph of $C_4$ and $H_2$ is an induced subgraph of $2K_2$, in which case $c \leq 1$;
        \item $H_1$ is a complete graph and $H_2$ is edgeless.
    \end{enumerate}
\end{theorem}

\begin{proof}
    We first show that if $H_1, H_2$ satisfy one of the cases, then $\{H_1, H_2\}$-free graphs are bounded-apex-split.
    \begin{enumerate}
        \item Follows from \cref{thm:apex-split-H-free}. 
        \item If $H_1$ is an induced subgraph of $C_4$ and $H_2$ is an induced subgraph of $2K_2$, then every graph $G$ that is $\{H_1, H_2\}$-free is pseudo-split by  \cref{thm:PseudoSplit} and 1-apex-split by \cref{thm:2K2-C4-free-structure}.
        \item Once again, this is immediate from Ramsey's Theorem, \cref{thm:ramsey}.
    \end{enumerate}
        
    Conversely, suppose that the class of $\{H_1, H_2\}$-free graphs is $c$-apex-split for some $c \in \mathbb{N}$. 
    Since split graphs are a subclass of chordal graphs, it suffices to consider the four cases of \cref{thm:apex-chordal_characterization}:
    \begin{enumerate}
        \item If $H_1$ is a proper induced subgraph of $P_3$, it is an induced subgraph of $K_2$ or $2K_1$. This corresponds to Case 1 of \cref{thm:apexsplit}. If $H_1 \cong P_3$, then $H_1$-free graphs are exactly disjoint unions of complete graphs. Accordingly, $(c+2)K_2$ and $2K_{c+2}$ are both $H_1$-free but not $c$-apex-split. The graph $H_2$ must therefore be contained in both $(c+2)K_2$ and $2K_{c+2}$. We conclude that $H_2$ is an induced subgraph of $2K_2$, which corresponds to Case 2 of \cref{thm:apexsplit}.
            
        \item For the case where $H_1$ is an induced subgraph of $P_4$ and $H_2 \cong C_4$, note that the graph $(c+2)K_2$ is $H_2$-free but not $c$-apex-split. Thus, $H_1$ must be an induced subgraph of $(c+2)K_2$ as well as of $P_4$. We conclude that $H_1$ must be contained in $2K_2$. This corresponds to Case 2 of \cref{thm:apexsplit}, after switching the labels of $H_1$ and $H_2$.

        \item The case with $H_1 \cong 2K_2$ and $H_2 \cong C_4$ corresponds to Case 2 of \cref{thm:apexsplit}.
            
        \item The case with $H_1$ being a complete graph and $H_2$ being edgeless corresponds to Case 3 of \cref{thm:apexsplit}.
    \end{enumerate}
    This completes the forward direction of the proof. The result follows.
\end{proof}

\section{Bounded-apex-bipartite graph classes}
In this section, we answer \cref{question:mainquestion} for the class of bipartite graphs. We first consider classes defined by a single forbidden induced subgraph.

\begin{theorem} \label{thm:1bip}
    Let $H$ be a graph. The class of $H$-free graphs is bounded-apex-bipartite if and only if $H$ is an induced subgraph of $K_2$.
\end{theorem}
\begin{proof}
    We first observe that if $H$ is an induced subgraph of $K_2$, each $H$-free graph is edgeless, and thus trivially bipartite. Hence it suffices to show the forward direction of the lemma statement.

     Let $c \in \mathbb{N}$ be given such that the class of $H$-free graphs is $c$-apex-bipartite. Let us first consider $K_{c+3}$. Deleting $c$ vertices leaves $K_3$, which is not bipartite; and hence $K_{c+3}$ is not $c$-apex-bipartite. It follows that $H$ is a complete graph. Since the graph $(c+1)C_5$ is not $c$-apex-bipartite, we conclude that $H$ is an induced subgraph of $K_2$, as desired.
\end{proof}

Next, we consider graph classes defined by two forbidden induced subgraphs.

\begin{theorem}
    Let $H_1$ and $H_2$ be graphs. The class of $\{H_1, H_2\}$-free graphs is bounded-apex-bipartite if and only if -- up to switching the labels of $H_1$ and $H_2$ -- one of the following holds:
    \begin{enumerate}
        \item $H_1$ is an induced subgraph of $K_2$; 
        \item $H_1 \cong K_3$ and $H_2$ is an induced subgraph of $P_3 + rK_1$ or $P_4$; or
        \item $H_1$ is a complete graph and $H_2$ is edgeless. 
    \end{enumerate}
\end{theorem}
\begin{proof}
We first show that if $H_1, H_2$ satisfy one of the cases, then $\{H_1, H_2\}$-free graphs are bounded-apex-bipartite. 
    \begin{enumerate}
        \item Follows from \cref{thm:1bip}. 
        \item Note that $P_4$-free graphs are perfect and hence $C_\ell$-free for all odd $\ell \geq 5$, so $\{K_3, P_4\}$-free graphs are bipartite. Likewise, $\{K_3, P_3+rK_1\}$-free graphs are bounded-apex-bipartite by \cref{lem:K_3-P_3+rK_1-free-is-c-apex-perfect}. 
        \item Once again, this is immediate from Ramsey's Theorem, \cref{thm:ramsey}.
    \end{enumerate}

    Now suppose that $\{H_1, H_2\}$-free graphs are $c$-apex-bipartite for some $c \in \mathbb{N}$. Since the class of complete graphs is not bounded-apex-bipartite, it follows that one of $H_1$ and $H_2$ is a complete graph; by symmetry, we assume that $H_1 \cong K_t$ for some $t \in \mathbb{N}$. If $t \leq 2$, then Case 1 holds.  

   Therefore, we  assume that $H_1$ contains $K_3$ as an induced subgraph. Since the graphs $(c+1)C_{2k+1}$ for all $c, k \in \mathbb{N}$ are $K_3$-free and not $c$-apex-bipartite, it follows that $H_2$ is an induced subgraph of each such graph, and hence $H_2$ is a linear forest. As the family $F_n$ from \cref{lem:blowupC5} is $K_3$-free and not $c$-apex-bipartite, it follows that $H_2$ is an induced subgraph of $P_4$ or $P_3 + rK_1$.

   Suppose first that $H_1 \cong K_3$. Then Case 2 holds. Thus we assume that $H_1$ contains $K_4$. Now consider the complete tripartite graph $K_{c+1, c+1, c+1}$, which is not $c$-apex-bipartite by virtue of containing $c+1$ disjoint copies of $K_3$, and hence, $K_{c+1, c+1, c+1}$ contains $H_2$. Since $H_2$ is $K_3$-free, it follows that $H_2$ is an induced subgraph of $K_{c+1, c+1}$. It follows that $H_2$ is one of: 
   \begin{itemize}
       \item $rK_1$ for some $r \in \mathbb{N}$ (so Case 3 holds); 
       \item an induced subgraph of $P_3$.
   \end{itemize}
   If $H_2$ is an induced subgraph of $P_3$ but not $P_3$ itself, then either $H_2$ is an induced subgraph of $K_2$ (and Case 1 holds) or $H_2 \cong 2K_1$ (and Case 2 holds). We conclude that $H_2 \cong P_3$. However, now the graph $(c+1)K_3$ is $\{H_1, H_2\}$-free and not $c$-apex-bipartite, a contradiction. 
\end{proof}

\section{Bounded-apex-complete-multipartite graph classes}
In this section, we answer \cref{question:mainquestion} for complete-multipartite graphs. We first consider classes defined by a single forbidden induced subgraph.

\begin{theorem} \label{thm:1mp}
    Let $H$ be a graph. The class of $H$-free graphs is bounded-apex-complete-multipartite if and only if $H$ is an induced subgraph of $K_2 + K_1$.
\end{theorem}
\begin{proof}
    Note that the class of $(K_2 + K_1)$-free graphs is exactly the class of complete multipartite graphs. Hence, if $H$ is an induced subgraph of $K_2+K_1$, the class of $H$-free graphs is complete multipartite. It thus suffices to show the forward direction of the lemma statement.

    Let $c \in \mathbb{N}$ be given such that the class of $H$-free graphs is $c$-apex-complete-multipartite. Since the graph $2K_{c+2}$ is not $c$-apex-complete-multipartite, as removing at most $c$ vertices guarantees the existence of $K_2 + K_1$, it contains $H$. We thus find that $H$ is the disjoint union of two cliques. 
    Next, consider the graph $(c+2)K_2$. This graph is not $c$-apex-complete-multipartite, as there does not exist a set of $c$ vertices hitting all copies of $K_2 + K_1$, and thus it also contains $H$. It follows that $H$ is an induced subgraph of $2K_2$. Finally, the class of $2K_2$-free graphs itself is not $c$-apex-complete-multipartite, as it contains the stable-set blowup of $C_5$ where each vertex is replaced by a stable set of size $c+1$. Thus $H$ is a proper induced subgraph of $2K_2$ and thus an induced subgraph of $K_2 + K_1$.
\end{proof}

Next, we consider graph classes defined by two forbidden induced subgraphs.

\begin{theorem}\label{thm:apex-complete-multipartite}
Let $H_1, H_2$ be graphs. Then the class of $\{H_1, H_2\}$-free graphs is bounded-apex-complete-multipartite if and only if one of the following holds (up to switching the names of $H_1$ and $H_2$):  
\begin{enumerate}
    \item $H_1$ is an induced subgraph of $K_2 + K_1$; or
    \item $H_1$ is a complete graph and $H_2$ is edgeless. 
\end{enumerate}
\end{theorem}
\begin{proof}
    We first show that if $H_1, H_2$ satisfy one of the cases above, then the class of $\{H_1, H_2\}$-free graphs is bounded-apex-complete-multipartite: 
    \begin{enumerate}
        \item This follows from \cref{thm:1mp}. 
        \item Once again, this is immediate from Ramsey's Theorem, \cref{thm:ramsey}. 
    \end{enumerate}

    For the forward direction, suppose that the class of $\{H_1, H_2\}$-free graphs is $c$-apex-complete-multipartite. Since the graph $(c+2)K_2$ is not $c$-apex-complete-multipartite, we assume without loss of generality that it contains $H_1$, that is $H_1 \cong rK_2 + sK_1$ for some $r, s \in \mathbb{N}$.
    
    Consider the graph $2K_{c+2}$, which is not $c$-apex-complete-multipartite and thus contains $H_1$ or $H_2$. We distinguish these two cases. Suppose first that $H_1$ is an induced subgraph of $2K_{c+2}$. Then, either $H_1$ is an induced subgraph of $K_2 + K_1$, which corresponds to Case 1 of \cref{thm:apex-complete-multipartite}, or $H_1 \cong 2K_2$. 
    Now consider the stable-set blowup of $C_5$ where each vertex is replaced by a stable set of size $c+1$. Since this graph is $2K_2$-free and not $c$-apex-complete multipartite, it contains $H_2$, and hence $H_2$ is triangle-free. Next, consider $K_{c+2} + (c+2)K_1$. This graph is $2K_2$-free and not $c$-apex-complete-multipartite, so it contains $H_2$. Since $H_2$ is triangle-free, it follows that $H_2 \cong K_2 + tK_1$ for some $t \in \mathbb{N}$. Finally, we consider $\overline{(c+1)C_5}$, which is $2K_2$-free and has no three-vertex stable set, and is not $c$-apex-complete-multipartite. It follows that $H_2$ is an induced subgraph of $K_2 + K_1$, which corresponds to Case 1 of \cref{thm:apex-complete-multipartite}. 

    In the second case, $H_2$ is an induced subgraph of $2K_{c+2}$, and so $H_2 \cong K_{\ell} + K_m$ for some $\ell,m \in \mathbb{N}$ with $\ell \geq m$. 
    Since the (symmetric) case where $H_1$ is an induced subgraph of $2K_2$ is addressed above, we assume that $\ell \geq 3$. 
    Consider the graph $G$ obtained from $K_{c+2, c+2}$ by deleting a perfect matching. This graph is not $c$-apex-complete-multipartite, but it is $K_3$-free and hence $H_2$-free, and thus $G$ contains $H_1$. It follows that either $H_1 \cong sK_1$, or $H_1 \cong K_2 + K_1$ (which corresponds to Case 1), or $H_1 \cong 2K_2$ (which we already addressed). It remains to consider the case where $H_1 \cong sK_1$. We may assume that $H_1$ is not an induced subgraph of $K_2 + K_1$, and so $s \geq 3$. 

    If $H_2$ is a complete graph, then Case 2 holds, so we assume that $m \geq 1$. Now consider the non-$c$-apex-complete-multipartite graph $G' := \overline{(c+1)C_5}$ once again. Each vertex in this graph has at most 2 non-neighbours, so $G'$ is $H_2$-free. Also, since $(c+1)C_5$ is triangle-free, it follows that $G'$ is $H_1$-free. This is a contradiction, and concludes the proof.  
\end{proof}

\section{Open Questions}

We conclude the paper with open questions about bounded-apex classes. 

Though our paper addresses many subclasses of perfect graphs and related classes, we propose the further study of \cref{question:mainquestion} for other classes. Among the numerous hereditary graph classes worth considering, we would be most curious about weakly-chordal, comparability graphs, permutation graphs, and the class of line graphs of bipartite graphs. Another natural direction of further research is to return to \cref{question:broadquestion} and characterize when forbidding an arbitrary (possibly infinite) set of induced subgraphs renders a class bounded-apex-perfect (or bounded-apex-chordal, etc.). 

We now consider a generalization of the $c$-apex-$\mathcal{P}$ relation. Rather than removing a constant number of vertices, what happens if we remove a number of vertices determined by $\omega(G)$ or another graph parameter? Given a graph property $\mathcal{P}$, a function $f: \mathbb{R} \to \mathbb{N}$, and a graph parameter $a: \mathcal{G} \to \mathbb{R}$; we say that a graph class $\mathcal{F}$ is \emph{$f(a)$-apex-$\mathcal{P}$} if for all $G \in \mathcal{F}$, $G$ is $f(a(G))$-apex-$\mathcal{P}$. In the case where $f$ is the constant function $c$, this definition is equivalent to $\mathcal{F}$ being $c$-apex-$\mathcal{P}$. If such a function $f$ exists for a given parameter $a$, we say that $\mathcal{F}$ is \emph{$a$-bounded-apex-$\mathcal{P}$}. Let $\mathcal{G}$ be the set of all graphs. 

\begin{question}
    Given a property $\mathcal{P}$ and a graph parameter $a: \mathcal{G} \to \mathbb{R}$, for which sets of graphs $\mathcal{H}$ is the class of $\mathcal{H}$-free graphs $a$-bounded-apex-$\mathcal{P}$? 
\end{question}

We are particularly interested in this question for parameters $\omega(G)$ and $\chi(G)$ and for perfect graphs. Knowing whether a class of graphs is $\omega$-bounded-apex-perfect or $\chi$-bounded-apex-perfect, and if so for what sort of function $f$, provides a means of determining how close or far from perfect different hereditary classes are in a more local sense than measures like $\chi$-boundedness.


\printbibliography

@article{maffray1994linear,
  title={Linear recognition of pseudo-split graphs},
  author={Maffray, Fr{\'e}d{\'e}ric and Preissmann, Myriam},
  journal={Discrete Applied Mathematics},
  volume={52},
  number={3},
  pages={307},
  year={1994},
  publisher={Amsterdam: North Holland.},
  doi={10.1016/0166-218X(94)00022-0}
}

@incollection{grotschel1984polynomial,
  title={Polynomial algorithms for perfect graphs},
  author={Gr{\"o}tschel, Martin and Lov{\'a}sz, L{\'a}szl{\'o} and Schrijver, Alexander},
  booktitle={North-Holland mathematics studies},
  volume={88},
  pages={325--356},
  year={1984},
  publisher={Elsevier}
}

@article{berge1961farbung,
  title={F\"arbung von Graphen, deren s\"amtliche bzw. deren ungerade Kreise starr sind},
  author={Berge, Claude},
  journal={Wissenschaftliche Zeitschrift},
  year={1961},
  publisher={Martin Luther Universitat Halle-Wittenberg}
}

@article{lovasz1972characterization,
  title={A characterization of perfect graphs},
  author={Lov{\'a}sz, L{\'a}szl{\'o}},
  journal={Journal of Combinatorial Theory, Series B},
  volume={13},
  number={2},
  pages={95--98},
  year={1972},
  publisher={Elsevier},
  doi={10.1016/0095-8956(72)90045-7}
}

@incollection{ramsey1987problem,
  title={On a problem of formal logic},
  author={Ramsey, Frank P},
  booktitle={Classic Papers in Combinatorics},
  pages={1--24},
  year={1987},
  publisher={Springer},
  doi={10.1112/plms/s2-30.1.264},
}

@article{Lekkerkerker1962representation,
  title={Representation of a finite graph by a set of intervals on the real line},
  author={Lekkerkerker, Cornelis and Boland, Johan},
  journal={Fundamenta Mathematicae},
  volume={51},
  number={1},
  pages={45--64},
  year={1962}
}

@article{sau2023k,
  title={{$k$}-apices of minor-closed graph classes. {I}. {B}ounding the obstructions},
  author={Sau, Ignasi and Stamoulis, Giannos and Thilikos, Dimitrios M},
  journal={Journal of Combinatorial Theory, Series B},
  volume={161},
  pages={180--227},
  year={2023},
  publisher={Elsevier},
  doi={10.1016/j.jctb.2023.02.012}
}

@article{marx2012obtaining,
  title={Obtaining a planar graph by vertex deletion},
  author={Marx, D{\'a}niel and Schlotter, Ildik{\'o}},
  journal={Algorithmica},
  volume={62},
  number={3},
  pages={807--822},
  year={2012},
  publisher={Springer},
  doi={10.1007/s00453-010-9484-z},
}

@article{borowiecki2018p,
  title={{$\cal{P}$}-apex graphs},
  author={Borowiecki, Mieczys{\l}aw and Drgas-Burchardt, Ewa and Sidorowicz, El{\.z}bieta},
  journal={Discussiones Mathematicae Graph Theory},
  volume={38},
  number={2},
  pages={323--349},
  year={2018},
  doi={10.7151/dmgt.2041},
}

@article{singh2023apex,
  title={Apex graphs and cographs},
  author={Singh, Jagdeep and Sivaraman, Vaidy and Zaslavsky, Thomas},
  journal={Theory \& Applications of Graphs},
  volume={11},
  number={1},
  pages={4},
  year={2024},
  doi={10.20429/tag.2024.110104}
}

@article{brandstadt2002maximum,
  title={Maximum weight stable set on graphs without claw and co-claw (and similar graph classes) can be solved in linear time},
  author={Brandst{\"a}dt, Andreas and Mahfud, Suhail},
  journal={Information Processing Letters},
  volume={84},
  number={5},
  pages={251--259},
  year={2002},
  publisher={Elsevier},
  doi={10.1016/S0020-0190(02)00291-0},
}

@article{chudnovsky2006strong,
  title={The strong perfect graph theorem},
  author={Chudnovsky, Maria and Robertson, Neil and Seymour, Paul and Thomas, Robin},
  journal={Annals of mathematics},
  pages={51--229},
  year={2006},
  publisher={JSTOR},
  doi={10.4007/annals.2006.164.51}
}

@article{olariu1988pawfree,
    title = {Paw-free graphs},
    author = {Stephan Olariu},
    journal = {Information Processing Letters},
    volume = {28},
    number = {1},
    pages = {53-54},
    year = {1988},
    doi={10.1016/0020-0190(88)90143-3},
}

@article{BLAZSIK199351,
title = {Graphs with no induced {${C}_4$} and {$2{K}_2$}},
journal = {Discrete Mathematics},
volume = {115},
number = {1},
pages = {51-55},
year = {1993},
issn = {0012-365X},
doi = {10.1016/0012-365X(93)90477-B},
url = {https://doi.org/10.1016/0012-365X(93)90477-B},
author = {Zoltán Blázsik and Mihály Hujter and András Pluhár and Zsolt Tuza},
}

@article{split-graph-forbidden-induced-subgraphs,
    author = {Stéphane Foldes and Peter L. Hammer},
    title = {Split graphs},
    journal = {Congressus Numerantium},
    year = {1977},
    pages = {311--315},
    volume={No. XIX}
    
}

@article{trotignon2013perfect,
  title={Perfect graphs: a survey},
  author={Trotignon, Nicolas},
  journal={arXiv preprint arXiv:1301.5149},
  year={2013},
  doi={10.48550/arXiv.1301.5149}
}

@book{golumbic2004algorithmic,
  title={Algorithmic graph theory and perfect graphs},
  author={Golumbic, Martin Charles},
  volume={57},
  year={2004},
  publisher={Elsevier},
  doi={10.1016/C2013-0-10739-8},
}

@inbook{chordal, place={Cambridge}, series={Encyclopedia of Mathematics and its Applications}, title={Chordal graphs}, booktitle={Topics in Algorithmic Graph Theory}, publisher={Cambridge University Press}, author={Golumbic, Martin Charles}, editor={Beineke, Lowell W. and Golumbic, Martin Charles and Wilson, Robin J.}, year={2021}, pages={130–151}, collection={Encyclopedia of Mathematics and its Applications}}

@article{interval,
    title = {{Interval graphs}},
    year = {2004},
    journal = {Annals of Discrete Mathematics},
    author = {Golumbic, Martin Charles},
    number = {C},
    month = {1},
    pages = {171--202},
    volume = {57},
    publisher = {Elsevier},
    url = {https://www.sciencedirect.com/science/article/abs/pii/S0167506004800566},
    doi = {10.1016/S0167-5060(04)80056-6},
    issn = {0167-5060}
}

@book{bipartite, place={Cambridge}, series={Cambridge Tracts in Mathematics}, title={Bipartite Graphs and their Applications}, publisher={Cambridge University Press}, author={Asratian, Armen S. and Denley, Tristan M. J. and Häggkvist, Roland}, year={1998}, collection={Cambridge Tracts in Mathematics}}

@inbook{split, place={Cambridge}, series={Encyclopedia of Mathematics and its Applications}, title={Split graphs}, booktitle={Topics in Algorithmic Graph Theory}, publisher={Cambridge University Press}, author={Collins, Karen L. and Trenk, Ann N.}, editor={Beineke, Lowell W. and Golumbic, Martin Charles and Wilson, Robin J.Editors}, year={2021}, pages={189–206}, collection={Encyclopedia of Mathematics and its Applications}}

\end{document}